\documentclass[12pt]{article}
\usepackage[top=3cm,bottom=3cm,left=3cm,right=3cm]{geometry}
\usepackage{graphicx}
\usepackage{diagbox}
\usepackage{color}
\usepackage{cite}
\usepackage{array}
\usepackage{amsthm,amsmath,amssymb,indentfirst}
\usepackage{arydshln}
\usepackage{rotating}
\usepackage{dashrule}
\usepackage{epstopdf}
\usepackage{setspace}
\usepackage{amsfonts}   % $\Box$
\usepackage{multirow}
\usepackage{makecell}
\usepackage{indentfirst} % 导入indentfirst包
\usepackage{proof}
\graphicspath{{figure/}}

\newtheorem{theorem}{Theorem}[section]
\newtheorem{example}{Example}[section]

\newtheorem{lemma}[theorem]{Lemma}

\newtheorem{corollary}{Corollary}[section]

\date{}
\begin{document}
\baselineskip=0.8truecm
\begin{center}\setlength{\topskip}{0cm}\huge\Large \textbf{The local antimagic (total) chromatic numbers of firecracker graphs and edge-corona product
graphs}
\normalsize \\ [6truemm] Xue Yang$^{a}$, Hong Bian$^{b}$\footnotemark[2], Xueliang Li$^{b}$, Zhixia Yang$^{a}$, Haizheng Yu$^{a}$\\
\footnotetext[2]{Corresponding author: bh1218@163.com}
{\small $^a$ College of Mathematics and System Sciences, Xinjiang University,}\\{\small Urumqi, Xinjiang 830046, P. R. China}\\
{\small$^b$ School of Mathematical Sciences Xinjiang Normal University,}\\{\small Urumqi, Xinjiang 830054, P. R. China}\\

\end{center}
\normalsize\thispagestyle{empty}

\noindent{\sl\bf Abstract:} Let $G=(V(G), E(G))$ be a connected simple graph with $n$ vertices and $m$ edges. A bijection $f: E(G)\rightarrow \{1,2,\cdots,m\}$ is called a local antimagic labeling of $G$,
if for any two adjacent vertices $u$ and $v$ in $G$ we have $\omega(u)\neq\omega(v)$, where $\omega(u)=\sum_{e\in E(u)}f(e)$, and $E(u)$ is the set of edges incident with $u$.
Similarly, a bijection $g: V(G)\cup E(G)\rightarrow \{1,2,\cdots,n+m\}$ is called a local antimagic total labeling of $G$, if for any two adjacent vertices $u$ and $v$ in $G$
we have $\omega_{t}(u)\neq\omega_{t}(v)$, where $\omega_{t}(u)=g(u)+\sum_{e\in E(u)}g(e)$. Obviously, any local antimagic (total) labeling induces a proper vertex-coloring of $G$
when every vertex $v$ is assigned the color $\omega(v)$ ($\omega_{t}(v)$). The local antimagic (total) chromatic number of $G$, denoted by $\chi_{la}(G)$ ($\chi_{lat}(G)$),
is defined as the minimum number of colors taken over all colorings induced by local antimagic (total) labelings of $G$. In this paper, we present the local antimagic (total) chromatic number
of firecracker graph $F_{n,k}$, obtained by the concatenation of $n$ $k$-stars by linking one leaf from each. Then we give the local antimagic chromatic number of the edge-corona product $G\diamond H$
of two graphs $G$ and $H$, where the graph $G\diamond H$ is constructed by taking one copy of $G$ and $|E(G)|$ disjoint copies of $H$ one-to-one assigned to each edge of $G$,
and for every edge $uv \in E(G)$, joining $u$ and $v$ to every vertex of the copy of $H$ associated to $uv$. For the graph $G\diamond H$ studied here, $G$ is a star $S_{k}$ or a
double star $S_{k_{1},k_{2}}$, and $H$ is an empty graph $\overline{K_{r}}$ or a complete graph $K_{2}$.\\

\baselineskip=0.30in

\section{Introduction}

Given a connected simple graph $G=(V(G),E(G))$ with $n$ vertices and $m$ edges, Arumugam et al. \cite{Arumugam-1} and Bensmail et al. \cite{Bensmail} independently introduced
the definition of local antimagic labeling of $G$, which goes as follows. A bijection $f: E(G)\rightarrow\{1,2, \cdots, m\}$ is called a local antimagic labeling of $G$
if any two adjacent vertices $u$ and $v$ in $G$ satisfy that $\omega(u) \neq \omega(v)$, where $\omega(u)=\sum_{e\in E(u)}f(e)$, and $E(u)$ denotes the set of edges incident with vertex $u$.
Putri et al. \cite{Putri} extended this concept to local antimagic total labeling of graphs. A bijection $g: V(G)\cup E(G)\rightarrow\{1,2, \cdots, n+m\}$ is called
a local antimagic total labeling of $G$ if any two adjacent vertices $u$ and $v$ in $G$ satisfy that $\omega_{t}(u) \neq \omega_{t}(v)$, where $\omega_{t}(u)=g(u)+\sum_{e\in E(u)}g(e)$.
It is evident that by assigning $\omega(x)$ ($\omega_{t}(x)$) as the color of the vertex $x$ for each $x\in V(G)$, a proper vertex-coloring of $G$ is naturally induced, which is called a local antimagic (total) coloring labeling of $G$. 
The local antimagic (total) chromatic number \cite{Arumugam-1,Putri} of $G$, denoted by $\chi_{la}(G)$ ($\chi_{lat}(G)$), is defined as the minimum number of colors taken
over all the vertex-colorings of $G$ induced by the local antimagic (total) labelings of $G$. By the definition of the local antimagic (total) chromatic number,
it follows that $\chi_{la}(G)\geq \chi(G)$ ($\chi_{lat}(G)\geq \chi(G)$).

%Hartsfield and Ringel \cite{Hartsfield} introduced the concept of antimagic. A graph $G$ is called antimagic if $G$ has an antimagic labeling.
%A bijective $f: E(G)\rightarrow\{1,..., m\}$ is called an antimagic labeling of $G$ if for any two distinct vertices $u$ and $v$ of $G$, $\omega(u)=\omega(v)$, where $\omega(u)=\sum f(e)$ with $e$ ranging over all the edges incident to $u$.
%The most known unsolved problems encompass two conjectures. The one conjecture is that every connected graph other than $K_{2}$ is antimagic.
%The other conjecture is that every tree other than $K_{2}$ is antimagic.
Arumugam et al. in \cite{Arumugam-1} conjectured that every connected graph other than $K_{2}$ is local antimagic. They thought that it would be difficult to solve the conjecture for general connected graphs. So, they proposed a weaker conjecture that every tree other than $K_{2}$ is local antimagic. Later,
Haslegrave in \cite{Haslegrave} confirmed the conjecture in general, and moreover, determined the local antimagic chromatic number for every connected graph except $K_{2}$.
Arumugam et al. in \cite{Arumugam-1} also derived the upper and lower bound for $\chi_{la}(G \vee \overline{K_{2}})$, where $\overline{K_{2}}$ is the empty graph with only 2 vertices, or the complement of the complete graph $K_{2}$.
Then, Lau et al. in \cite{Lau-1} provided counterexamples to the lower bound of $\chi_{la}(G \vee \overline{K_{2}})$. Another counterexample was independently discovered by Shaebani in \cite{Shaebani}.
Arumugam et al. in \cite{Arumugam-3} gave the local antimagic chromatic number of several families of trees. The local antimagic chromatic number for the join of some graphs were determined in \cite{Lau-1,Lau-2,Lau-3,Lau-4,Lau-6,Yang-1,Yang-2,Premalatha}.
Concerning other kinds of graphic operations of two graphs, the result on lexicographic product of graphs $G$ and $\overline{K_{m}}$, is obtained by Lau et al. in \cite{Lau-5}.
Some scholars investigated the local antimagic chromatic numbers of some special graphs, including unicyclic graphs \cite{Nazula}, disjoint union of multiple copies of a graph \cite{Baca}, graph with certain number of pendants \cite{Lau-8}, spider graph \cite{Lau-9} and complete full $t$-ary trees \cite{Baca-2}.
The local antimagic total chromatic number has been examined for some graphs, including families of trees \cite{Putri}, graph $G\circ K_{2}$ \cite{Kurniawati-1}, graph $G\circ mK_{2}$ \cite{Kurniawati-2}, and other graphs \cite{Lau-7} such as bipartite graph, path, or Cartesian product of two cycles.
%the join graph \( G \) with a complete graph \( K_{2} \) as detailed in \cite{Kurniawati-1}, and the join of \( G \) with multiple \( K_{2} \) subgraphs as discussed in \cite{Kurniawati-2}. Furthermore, the chromatic numbers for bipartite graphs, paths, and the Cartesian product of two cycles have been explored in \cite{Lau-7}. This research provides a nuanced perspective on the local antimagic total chromatic properties of these diverse graph structures.
%The local antimagic total chromatic numbers of graphs can be seen, such as some families trees \cite{Putri}, the graph $G\circ K_{2}$ \cite{Kurniawati-1},  the graph $G\circ mK_{2}$ \cite{Kurniawati-2},  and the graphs \cite{Lau-7} that bipartite graphs, paths,  or the Cartesian product of two cycles.

A $k$-star, denoted by $S_{k}$, has order $k+1$ and size $k$, and $k$ pendant vertices.
The $(n,\,k)$-firecracker graph, denoted by $F_{n,k}$, is constructed by concatenating $n$ $k$-stars by linking one leaf from each. The edge-corona product of two graphs $G$ and $H$, denoted by $G\diamond H$, is formed by taking one copy of $G$ and $|E(G)|$ disjoint copies of $H$, with each copy assigned to an edge of $G$. For every edge $uv \in E(G)$, vertices $u$ and $v$ are joined to every vertex in the copy of $H$ associated with $uv$. For further details, see \cite{Hou,Yeh}. In this paper, we determine the local antimagic (total) chromatic number of firecracker graph $F_{n,k}$. Additionally, we provide the local antimagic chromatic number for edge-corona product graphs $G\diamond H$, where $G$ is a star $S_{k}$ or a double star $S_{k_{1},k_{2}}$, and $H$ is the empty graph $\overline{K_{r}}$ or the complete graph $K_{2}$.

\section{Preliminaries}

In this section we will list some notations to be used in this paper. Vectors are denoted by lowercase bold letters, and matrices are denoted by uppercase bold letters. Let $\boldsymbol{e}$ and $\boldsymbol{I}$ be the row vector and the diagonal matrix whose all entries are 1, respectively. The diagonal matrix with entries increasing and starting at 1 is denoted by $\boldsymbol{\Omega}$, and the null matrix with all entries the symbol $*$ is denoted by $\bigstar$. The space of all matrices with $m$ rows and $n$ columns is denoted by $\mathbb{R}^{m\times n}$. For the positive integers $a<b$, denote $[a,b]$ as the set of integers $\{c\in \mathbb{Z}|a\leq c\leq b\}$.%[a,b]=
%$\mathbb{R}^{m\times n}$ is denoted as the space of matrices with  rows and $n$ columns.

The firecracker graph $F_{n,k}$ is a graph of order $n+nk$ and size $n+nk-1$. At first we show you an example. The firecracker graph $F_{5,4}$ is depicted in Fig.\,\ref{fig1}. In general, the firecracker graph $F_{n,k}$ is given by $V(F_{n,k})=\{u_{i}|1\leq i\leq n\}\cup\{v_{i,j}|1\leq i\leq n,\,1\leq j\leq k\}$ and $E(F_{n,k})=\{v_{i,1}v_{i+1,1}|1\leq i\leq n-1\}\cup\{u_{i}v_{i,j}|1\leq i\leq n,\,1\leq j\leq k\}$ as the vertex set and edge set, respectively.

\begin{figure}[htbp]
\centering
\includegraphics[width=120mm]{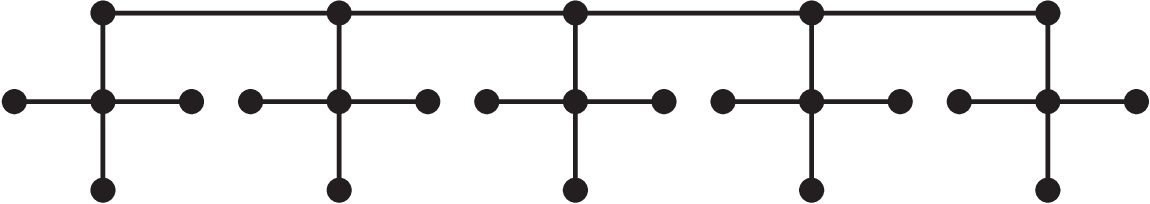}
\renewcommand{\figurename}{Fig.}
\caption{The firecracker graph $F_{5,4}$.}\label{fig1}
\end{figure}

For any two graphs $G$ and $H$, the edge-corona product graph $G\diamond H$ \cite{Hou,Yeh} is obtained by taking one copy of $G$ and $|E(G)|$ disjoint copies of $H$ one-to-one assigned to the edges of $G$, and for every edge $uv\in E(G)$ joining $u$ and $v$ to every vertex of the copy of $H$ associated to $uv$. As an example, the edge-corona product graph $P_{5}\diamond \overline{K_{3}}$ is showed in Fig.\,\ref{fig5}.

\begin{figure}[htbp]
	\centering
	\includegraphics[width=110mm]{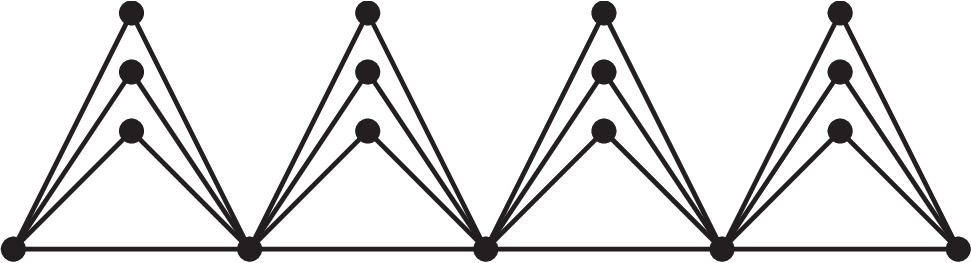}
	\renewcommand{\figurename}{Fig.}
	\caption{The edge-corona product graph $P_{5}\diamond \overline{K_{3}}$.}\label{fig5}
\end{figure}

%Let $f$ be a edge labeling of  graph $G$, then got the labeling of the edge labeling introduced in \cite{Gallian}.
Given a graph $G=(V(G),E(G))$ with $n$ vertices and $m$ edges, define $f: E(G)\rightarrow [1,m]$, where $f$ is a edge labeling of $G$. A labeling matrix \cite{Gallian} $\boldsymbol{\mathcal{M}}$ of $f$ for $G$ is an $n\times n$ symmetric matrix in which the $(i,j)$-entry of $\boldsymbol{\mathcal{M}}$ is $f(u_{i}u_{j})$ if $u_{i}u_{j}\in E(G)$ and is $*$ otherwise, and the labels are the numbers in the set $[1, m]$ and appear only once in the upper triangle of this labeling matrix. %If $f$ is a local antimagic labeling (shortened to {\it l.a.l.}) of $G$, then a labeling matrix of $f$ is called a local antimagic labeling matrix of $G$. Clearly, the $i$-th row sum, $\mathcal{R}_{i}$ (and $i$-th column sum) is $\omega_{f}(u_{i})$, where $*$ is treated as zero. Thus the condition of a labeling matrix being a local antimagic labeling matrix is the $i$-th row sum different the $i$-th row sum when $u_{i}u_{j}\in E(G)$. Then present the {\it l.a.c.n.} of the edge corona product of two graphs.

%The edge corona product of graphs $G$ and $H$, denoted by $G\diamond H$, is obtained by taking one copy of $G$ and $|E(G)|$ disjoint copies of $H$ one-to-one assigned to the edges of $G$, and for every edge $gg^{'} \in E(G)$ joining $g$ and $g'$ to every vertex of the copy of $H$ associated to $gg'$, see \cite{Hou,Yeh}.

Arumugam et al. \cite{Arumugam-1} got a lower bound of the local antimagic chromatic number for any tree $T$.
\begin{lemma}\cite{Arumugam-1}\,\label{th} For a tree $T$ with $l$ leafs, we have $\chi_{la}(T)\geq l+1$.
\end{lemma}

\section{Main results}

This section determines the local antimagic (total) chromatic number of firecracker graph $F_{n,k}$. Subsequently, we proceed to present the local antimagic chromatic number of edge-corona product graph $G\diamond H$, where $G$ is a star $S_{k}$ and a double star $S_{k_{1},k_{2}}$, and $H$ is the empty graph $\overline{K_{r}}$ and the complete graph $K_{2}$.

\subsection{The local antimagic (total) chromatic number of firecracker graph $F_{n,k}$}

%The Local Antimagic Chromatic Number of the Firecracker Graph
%$\chi_{la}(F_{n,k})$ and $\chi_{lat}(F_{n,k})$
Consider the firecracker graph $F_{n,k}$.
For $n=1$, it is a star, yielding that $\chi_{la}(S_{k})=k+1$ from Arumugam et al. \cite{Arumugam-1}.
When $k=1$, it becomes the corona-product graph of the path $P_{n}$ and empty graph $\overline{K_{1}}$, denoted by $P_{n}\circ \overline{K_{1}}$, with $\chi_{la}(P_{n}\circ \overline{K_{1}})=n+2$ from \cite{Arumugam-3}. Therefore, we focus on the local antimagic chromatic number of $F_{n,k}$ for $n\geq 2,\,k\geq 2$.%by Arumugam et al.

\begin{theorem}
For firecracker graph $F_{n,k}$ with $n\geq 2,\,k\geq 2$, the local antimagic chromatic number of $F_{n,k}$ is
\begin{equation*}
\chi_{la}(F_{n,k})=nk-n+1.
\end{equation*}
\end{theorem}

\begin{proof}
It is easy to see that firecracker graph $F_{n,k}$ has $n(k-1)$ leafs. By Lemma \ref{th}, we establish that $\chi_{la}(F_{n,k})\geq n(k-1)+1$. To determine the precise value of $\chi_{la}(F_{n,k})$, the next step is to obtain a local antimagic labeling $f$ of $F_{n,k}$,
such that $f$ uses exactly $nk-n+1$ colors.%the minimum number of colors taken all over all colorings induced by the local antimagic labelings of $F_{n,k}$.

Define a bijection $f:E(F_{n,k})\rightarrow [1,nk+n-1]$. Consider the following two cases based on the parity of $k$.

Case 1. $k$ is odd.

This case is further divided into two subcases based on the parity of $n$.

Subcase 1.1. $n$ is odd.

For $n=3$, we label the edges of the subgraph $F_{3,3}$ of $F_{3,k}$ $(k\geq 3)$ as follows:
\begin{equation*}
\begin{array}{llll}
f(u_{1}v_{1,1})=3, &f(u_{1}v_{1,2})=7,
&f(u_{1}v_{1,3})=11, &f(v_{1,1}v_{2,1})=2,\\[3pt]
f(u_{2}v_{2,1})=6, &f(u_{2}v_{2,2})=5,
&f(u_{2}v_{2,3})=10,&f(v_{2,1}v_{3,1})=1,\\[3pt]
f(u_{3}v_{3,1})=4, &f(u_{3}v_{3,2})=8,&f(u_{3}v_{3,3})=9.
%f(v_{1,1}v_{2,1})=2, &f(v_{2,1}v_{3,1})=1.
\end{array}
\end{equation*}
Next, label the edges $u_{i}v_{i,j}$ for $1\leq i\leq n$ and $4\leq j\leq k$ with
\begin{equation*}
f(u_{i}v_{i,j})=\left\{\begin{array}{ll}
3j-1+i,   &\text{if $j$ even},\\[3pt]
3j+3-i, &\text{if $j$ odd}.
\end{array}\right.
\end{equation*}
By the above labeling, vertices are colored by the sum of the labels of the incident edges. We find that $\omega(v_{1,1})=\omega(v_{3,1})=5$, $\omega(v_{2,1})=9$, $\omega(u_{i})=\frac{3k^{2}+5k}{2}>3k+2$ for $1\leq i\leq n$, and $5\leq\omega(v_{i,j})\leq 3k+2$ for $1\leq i\leq n$ and $2\leq j\leq k$.
The colors of $3k-2$ pendant vertices are distinct, and we have $\omega(v_{1,1})=\omega(v_{3,1})=\omega(v_{2,2})=5$,
$\omega(v_{2,1})=\omega(v_{3,3})=9$,
thus the local antimagic labeling $f$ utilizes $3k-2$ colors.
As a result, $\chi_{la}(F_{3,k})=3k-2$ for $k$ odd.
%\begin{equation*}
%\begin{array}{l}
%\omega(v_{1,1})=\omega(v_{3,1})=5,\hspace{2em}\omega(v_{2,1})=9,\\[3pt]
%\omega(u_{i})={\frac{3k^{2}+5k}{2}}>3k+2~\text{for}~1\leq i\leq n,\\[6pt]
%5\leq\omega(v_{i,j})\leq 3k+2~\text{for}~1\leq i\leq n~\text{and}~2\leq j\leq k.
%\end{array}
%\end{equation*}

If $n\geq 5$, the local antimagic chromatic number $\chi_{la}(F_{n,k})$ is computed using the following formula.
\begin{equation*}
\begin{array}{l}
f(v_{i,1}v_{i+1,1})=\left\{\begin{array}{ll}
\frac{n-i}{2},  &\text{if } 1\leq i\leq n-1~\text{and}~i~\text{odd},\\[4pt]
n-\frac{i}{2}, &\text{if } 1\leq i\leq n-1~\text{and}~i~\text{even}.
\end{array}\right.\vspace{6pt}\\
f(u_{i}v_{i,1})=\left\{\begin{array}{ll}
2n-1, &\text{if } i=1,\\[1pt]
n-1+i,&\text{if } 2\leq i\leq n~\text{and}~i~\text{even},\\[1pt]
n-3+i,&\text{if } 2\leq i\leq n~\text{and}~i~\text{odd}.
\end{array}\right.
\end{array}
\end{equation*}
Subsequently, the edges $u_{i}v_{i,2}$ and $u_{i}v_{i,3}$ are labeled based on the labels of $u_{i}v_{i,1}$ as shown in Tab.\,1.
%in accordance with the column in the Tab.\,1. the edges $u_{i}v_{i,2}$ and $u_{i}v_{i,3}$ are progressively labeled based on the labels assigned to edges $u_{i}v_{i,1}$.
\begin{table}[htbp]\label{T1}
\renewcommand\arraystretch{1.8}
\renewcommand{\tablename}{Tab.}
\caption{The labels of edges $u_{i}v_{i,j}$ $(1\leq i\leq n,\,j=1,2,3)$ in $F_{n,k}$}
\vspace{0.5em}
\centering
\scriptsize
\begin{tabular}{|c|c|c|c|c|c|c|c|c|c|c|c|c|c|}\hline
 % \multicolumn{14}{|c|}{{\normalsize{The values of $f(u_{i}v_{i,j})$ for $1\leq i\leq n,\,j=1,2,3$}}}\\\hline
 \diagbox{$j$}{$i$} &$1$&$2$&\multicolumn{2}{c|}{$\cdots$} &$\frac{n-3}{2}$ &$\frac{n-1}{2}$ &$\frac{n+1}{2}$&$\frac{n+3}{2}$&$\frac{n+5}{2}$&\multicolumn{2}{c|}{$\cdots$} &$n-1$ &$n$\\\hline
   $1$ &$n$ &$n+1$ &\multicolumn{2}{c|}{$\cdots$} &$\frac{3n-5}{2}$ &$\frac{3n-3}{2}$ &$\frac{3n-1}{2}$
   &$\frac{3n+1}{2}$  &$\frac{3n+3}{2}$
&\multicolumn{2}{c|}{$\cdots$} &$2n$ &$2n-1$\\\hline

$2$ &$3n-2$ &$3n-4$ &\multicolumn{2}{c|}{$\cdots$} &$2n+3$ &$2n+1$ &$3n-1$
   &$3n-3$  &$3n-5$
&\multicolumn{2}{c|}{$\cdots$} &$2n+2$ &$2n$\\\hline

$3$ &$\frac{7n+1}{2}$ &$\frac{7n+3}{2}$ &\multicolumn{2}{c|}{$\cdots$} &$4n-2$ &$4n-1$ &$3n$
   &$3n+1$  &$3n+2$
&\multicolumn{2}{c|}{$\cdots$} &$\frac{7n-3}{2}$ &$\frac{7n-1}{2}$\\\hline
\end{tabular}
\end{table}

\noindent Proceed to label the edges $u_{i}v_{i,j}$ for $1\leq i\leq n$ and $4\leq j\leq k$.
\begin{equation*}
f(u_{i}v_{i,j})=\left\{\begin{array}{ll}
jn-1+i, &\text{if } j~\text{even},\\[1pt]
jn+n-i, &\text{if } j~\text{odd}.
\end{array}\right.
\end{equation*}
Consequently, we have $\omega(u_{i})=\frac{k(kn+2n-1)}{2}>nk+n-1$ for $1\leq i\leq n$,  $\omega(v_{1,1})=\frac{5n-3}{2}$,  $\omega(v_{i,1})=\frac{5n-1}{2}$ for $2\leq i\leq n$ and even $i$, $\omega(v_{i,1})=\frac{5n-5}{2}$ for $2\leq i\leq n$ and odd $i$,  $2n\leq\omega(v_{i,j})\leq nk+n-1$ for $1\leq i\leq n$ and $2\leq j\leq k$.
It is evident that $\omega(v_{i,1})$ must coincide with some colors $\omega(v_{i,j})$ of pendant vertices. For each pair of adjacent vertices $v_{i,1}$ and $v_{i+1,1}$, distinct colors are used. This leads to the conclusion that $\chi_{la}(F_{n,k})\leq nk-n+1$ for $k$ odd and $n\geq 5$.

%\begin{equation*}
%\begin{array}{l}
%\omega(u_{i})=\frac{k(kn+2n-1)}{2}>nk+n-1\text{ for } 1\leq i\leq n.\vspace{6pt}\\
%\omega(v_{i,1})=\left\{\begin{array}{ll}
%\frac{5n-3}{2}, &\text{for } i=1,\\[6pt]
%\frac{5n-1}{2}, &\text{for } 2\leq i\leq n~\text{and}~i~\text{even},\\[6pt]
%\frac{5n-5}{2}, &\text{for } 2\leq i\leq n~\text{and}~i~\text{odd}.
%\end{array}\right.\\
%\\[-9pt]
%2n\leq\omega(v_{i,j})\leq nk+n-1, \hspace{1em}1\leq i\leq n~\text{and}~2\leq j\leq k.
%\end{array}
%\end{equation*}

Subcase 1.2. $n$ is even.

When $n=2$ and $k\geq 3$, the local antimagic labeling of $F_{2,k}$ depicted in Fig.\,\ref{fig4}, is derived by
\begin{equation*}
\begin{array}{l}
f(u_{1}v_{1,1})=1, \hspace{1em}f(u_{2}v_{2,1})=4, \hspace{1em}f(u_{1}v_{1,2})=5, \hspace{1em}f(u_{2}v_{2,2})=3,\\[3pt]
f(u_{1}v_{1,3})=7, \hspace{1em}f(u_{2}v_{2,1})=6, \hspace{1em}f(v_{1,1}v_{2,1})=2,\\[3pt]
f(u_{i}v_{i,j})=\left\{\begin{array}{ll}
jn-1+i, &\text{if }1\leq i\leq n,~4\leq j\leq k~\text{and}~j~\text{even},\\[3pt]
jn+n-i, &\text{if }1\leq i\leq n,~4\leq j\leq k~\text{and}~j~\text{odd}.
\end{array}\right.
\end{array}
\end{equation*}

\begin{figure}[htbp]
\centering
\includegraphics[height=60mm]{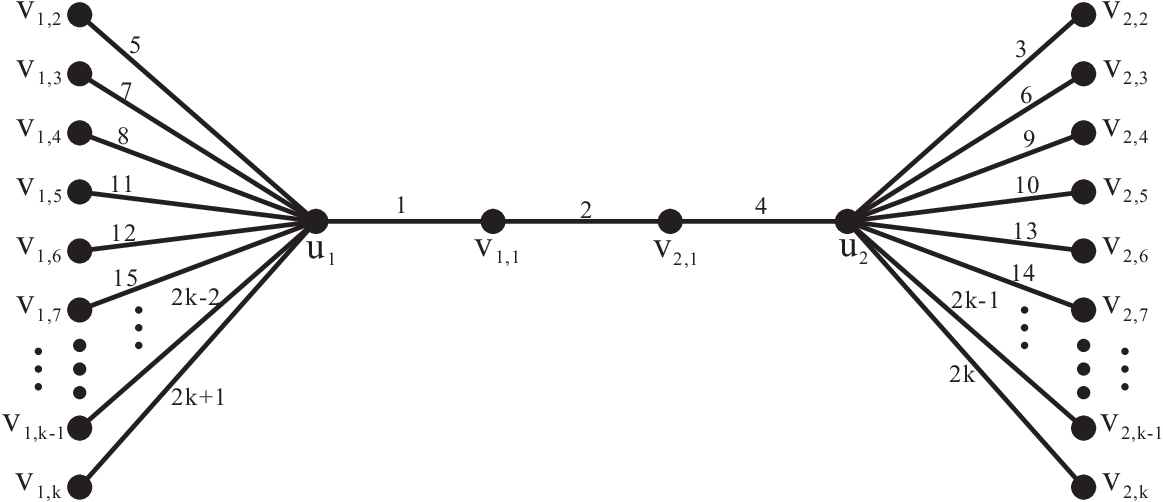}
\renewcommand{\figurename}{Fig.}
\caption{The local antimagic labeling of $F_{2,k}$ for $k\geq 3$.}\label{fig4}
\end{figure}
\noindent We immediately establish that $\omega(v_{1,1})=3,\,\omega(v_{2,1})=6$, $\omega(u_{1})=\omega(u_{2})=\frac{2k^{2}+3k-1}{2}$ and $3\leq \omega(v_{i,j})\leq 2k+1$ for $1\leq i\leq n,~2\leq j\leq k$.
It is demonstrated that $f$ is a local antimagic labeling using $2k-1$ distinct colors, thereby $\chi_{la}(F_{2,k})=2k-1$ for $k$ odd.

For $n\geq 4$, continue labeling the edges of $F_{n,k}$, where $n$ is even and $k$ odd, by
\begin{equation*}
f(v_{i,1}v_{i+1,1})=\left\{\begin{array}{ll}
{\frac{3n-2}{2}},  &\text{if }i=1,\\[3pt]
{\frac{2n-2-i}{2}}, &\text{if }2\leq i\leq n-1~\text{and}~i~\text{even}.\\[3pt]
{\frac{n-i+1}{2}}, &\text{if }2\leq i\leq n-1~\text{and}~i~\text{odd}.
\end{array}\right.
\end{equation*}
\begin{equation*}
f(u_{i}v_{i,j})=\left\{\begin{array}{ll}
{\frac{3n-4-2i}{2}}, &\text{if }1\leq i<\frac{n}{2}~\text{and}~j=1,\\[3pt]
{\frac{3n-4}{2}},   &\text{if }i=\frac{n}{2}~\text{and}~j=1,\\[3pt]
{\frac{5n-2-2i}{2}}, &\text{if }\frac{n}{2}<i<n~\text{and}~j=1,\\[3pt]
2n-1, &\text{if }i=n~\text{and}~j=1.\\
2n+1+2i, &\text{if }1\leq i<\frac{n}{2}~\text{and}~j=2,\\
2n+1,    &\text{if }i=\frac{n}{2}~\text{and}~j=2,\\
n+2i,    &\text{if }\frac{n}{2}<i<n~\text{and}~j=2,\\
2n,      &\text{if }i=n~\text{and}~j=2.\\
4n-1-i, &\text{if }1\leq i<\frac{n}{2},~\frac{n}{2}<i<n~\text{and}~j=3,\\
4n-1,    &\text{if }i=\frac{n}{2}~\text{and}~j=3,\\[1pt]
\frac{7n-2}{2},      &\text{if }i=n~\text{and}~j=3.\\[1pt]
jn-1+i, &\text{if }1\leq i\leq n,~4\leq j\leq k~\text{and}~j~\text{even},\\
jn+n-i, &\text{if }1\leq i\leq n,~4\leq j\leq k~\text{and}~j~\text{odd}.
\end{array}\right.
\end{equation*}
Accordingly, we obtain that $\omega(u_{i})=\frac{k^{2}n+2kn-k-1}{2}>nk+n-1$ for $1\leq i\leq n$,  $\omega(v_{1,1})=3n-4$, $\omega(v_{2,1})=4n-7$, $\omega(v_{i,1})=3n-2-2i$ for $2<i<\frac{n}{2}$, and $\omega(v_{i,1})=4n-1-2i$ for $\frac{n}{2}<i<n$, $\omega(v_{\frac{n}{2},1})=\frac{5n-4}{2}$,
$\omega(v_{n,1})=2n$, as well as $2n\leq \omega(v_{i,j})\leq nk+n-1$ for $1\leq i\leq n$ and $2\leq j\leq k$.
Thus, $\omega(v_{i,1})$ equals certain $\omega(v_{i,j})$, and the adjacent vertices $v_{i,1}$ receive distinct colors. This implies that $\chi_{la}(F_{n,k})\leq nk-n+1$ for $n$ even and $k$ odd.
%the value of $\omega(v_{i,1})$ must be equal to the value of certain $\omega(v_{i,j})$, and each two adjacent vertices $v_{i,1}$ (where $1\leq i\leq n$) is assigned a different color. Therefore, the local antimagic labeling $f$ has $nk-n+1$ colors, indicating $\chi_{la}(F_{n,k})\leq nk-n+1$ for $n$ even and $k$ odd.
%$\omega(v_{1,1})=$
%\begin{equation*}
%\begin{array}{l}
%\omega(u_{i})=\frac{k^{2}n+2kn-k-1}{2}>nk+n-1\text{ for }1\leq i\leq n,\vspace{5pt}\\
%\omega(v_{i,1})=\left\{\begin{array}{ll}
%3n-4, &\text{for }i=1,\\
%4n-7, &\text{for }i=2,\\
%3n-2-2i, &\text{for }2<i<\frac{n}{2},\\[1pt]
%\frac{5n-4}{2}, &\text{for }i=\frac{n}{2},\\[1pt]
%4n-1-2i, &\text{for }\frac{n}{2}<i<n,\\
%2n, &\text{for }i=n.
%\end{array}\right.\vspace{5pt}\\
%2n\leq \omega(v_{i,j})\leq nk+n-1\text{ for }1\leq i\leq n~\text{and}~2\leq j\leq k.
%\end{array}
%\end{equation*}

In summary, we conclude that $\chi_{la}(F_{n,k})=nk-n+1$ for odd $k$.

Case 2. $k$ is even.

The edges $\{v_{i,1}v_{i+1,1}|1\leq i\leq n-1\}\cup\{u_{i}v_{i,1},u_{i}v_{i,2}|1\leq i\leq n\}$ are initially labeled by $[1,3n-1]$. We further divide the discussion based on the parity of $n$ into two subcases.

Subcase 2.1. $n$ is odd.

When $n=3$ and $k=2$, the local antimagic labeling of $F_{3,2}$ is given in Fig.\,\ref{fig2}, i.e.,
\begin{equation*}
\begin{array}{llll}
f(u_{1}v_{1,1})=4, &f(u_{1}v_{1,2})=8,&f(u_{2}v_{2,1})=1, & f(u_{2}v_{2,2})=6,\\[3pt]
f(u_{3}v_{3,1})=5, &f(u_{3}v_{3,2})=7,&f(v_{1,1}v_{2,1})=3, &f(v_{2,1}v_{3,1})=2.
\end{array}
\end{equation*}
Hence, $\chi_{la}(F_{3,2})=4$.

\begin{figure}[htbp]
\centering
\includegraphics[width=70mm]{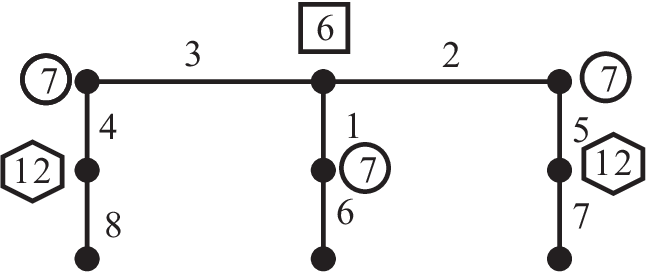}
\renewcommand{\figurename}{Fig.}
\caption{The local antimagic labeling of $F_{3,2}$.}\label{fig2}
\end{figure}

Considering $k\geq 4$, label the edges of the subgraph $F_{3,4}$ within $F_{3,k}$ in the following way,
\begin{equation*}
\begin{array}{l}
\begin{array}{llll}
f(u_{1}v_{1,1})=3, &f(u_{2}v_{2,1})=4,
&f(u_{3}v_{3,1})=8, &f(v_{1,1}v_{2,1})=2,\\[3pt]
f(u_{1}v_{1,2})=6, &f(u_{2}v_{2,2})=7,
&f(u_{3}v_{3,2})=5, &f(v_{2,1}v_{3,1})=1,\\[3pt]
f(u_{1}v_{1,3})=11, &f(u_{2}v_{2,3})=10,
&f(u_{3}v_{3,3})=9,&\\[3pt]
f(u_{1}v_{1,4})=14, &f(u_{2}v_{2,4})=13,
&f(u_{3}v_{3,2})=12. &
\end{array}\vspace{5pt}\\
f(u_{i}v_{i,j})=\left\{\begin{array}{ll}
3j-2+2i, &\text{if }1\leq i\leq n,\,5\leq j\leq k~\text{and}~j~\text{odd},\\
3j+4-2i, &\text{if }1\leq i\leq n,\,5\leq j\leq k~\text{and}~j~\text{even}.
\end{array}\right.
\end{array}
\end{equation*}
Then we get that $\omega(u_{i})=\frac{3k^{2}+6k-k}{2}$ for $1\leq i\leq n$, $\omega(v_{1,1})=5$, $\omega(v_{2,1})=7$, and $\omega(v_{3,1})=9$, as well as $5\leq \omega(v_{i,j})\leq 3k+2$ for $1\leq i\leq n$ and $2\leq j\leq k$.
It is observed that the local antimagic labeling $f$ utilizes $3k-2$ colors. Thus $\chi_{la}(F_{3,k})=3k-2$ when $k$ is even.

For $n\geq 5$, let's start labeling the edges $v_{i,1}v_{i+1,1}$ for $1\leq i\leq n-1$ by
\begin{equation*}
f(v_{i,1}v_{i+1,1})=\left\{\begin{array}{ll}
\frac{n-i}{2},  &\text{if $1\leq i\leq n-1$~and~$i$~odd},\\[3pt]
n-\frac{i}{2}, &\text{if $1\leq i\leq n-1$~and~$i$~even}.
\end{array}\right.
\end{equation*}
Then, the remaining edges are labeled by
\begin{equation*}
f(u_{i}v_{i,j})=\left\{\begin{array}{ll}
2n-1, &\text{if }i=1~\text{and}~j=1,\\
n-3+i,  &\text{if }3\leq i\leq n,~i~\text{odd~and}~j=1,\\
n-1+i,  &\text{if }2\leq i\leq n,~i~\text{even~and}~j=1,\\
2n, &\text{if }i=1~\text{and}~j=2,\\
3n+2-i,  &\text{if }3\leq i\leq n,~i~\text{odd~and}~j=2,\\
3n-i,    &\text{if }2\leq i\leq n,~i~\text{even~and}~j=2,\\
%3j-2+2i, &1\leq i\leq n,\,5\leq j\leq k~\text{and}~j~\text{odd},\\
%3j+4-2i, &1\leq i\leq n,\,5\leq j\leq k~\text{and
%}~j~\text{even}.
jn-2+2i, &\text{if $1\leq i\leq n,\,3\leq j\leq k$~and~$j$~odd},\\
(j+1)n+1-2i, &\text{if $1\leq i\leq n,\,3\leq j\leq k$~and~$j$~even}.
\end{array}\right.
\end{equation*}
From this, we find that $\omega(u_{i})=\frac{k^{2}n+2kn-k}{2}$ for $1\leq i\leq n$, $\omega(v_{1,1})=\frac{5n-3}{2}$, $\omega(v_{i,1})=\frac{5n-1}{2}$ for even $i$ and $\omega(v_{i,1})=\frac{5n-5}{2}$ for odd $i$ when $2\leq i\leq n$ , as well as $2n\leq \omega(v_{i,j})\leq nk+n-1$ for $1\leq i\leq n$ and $2\leq j\leq k$. Therefore, the local antimagic labeling $f$ uses $nk-n+1$ colors, and $\chi_{la}(F_{n,k})\leq nk-n+1$ for $n\geq 5$.

In this subcase, we have $\chi_{la}(F_{n,k})=nk-n+1$ for $n$ odd and $k$ even.
%\begin{equation*}
%\begin{array}{l}
%\omega(u_{i})=\frac{k^{2}n+2kn-k}{2}, \hspace{2em} 1\leq i\leq n,\vspace{5pt}\\
%\omega(v_{i,1})=\left\{\begin{array}{ll}
%\frac{5n-3}{2}, & i=1,\\[3pt]
%\frac{5n-1}{2}, &2\leq i\leq n~\text{and}~i~\text{even},\\[3pt]
%\frac{5n-5}{2}, &2\leq i\leq n~\text{and}~i~\text{odd}.
%\end{array}\right.\vspace{5pt}\\
%2n\leq \omega(v_{i,j})\leq nk+n-1, \hspace{1em}1\leq i\leq n~\text{and}~\,2\leq j\leq k.
%\end{array}
%\end{equation*}

Subcase 2.2. $n$ is even.

When $n=2$ and $k=2$, $F_{2,2}$ is also the path $P_{6}$. It follows that $\chi_{la}(F_{2,2})=\chi_{la}(P_{6})=3$ from Arumugam et al. \cite{Arumugam-1}. The local antimagic labeling of $F_{2,k}$ is shown in Fig. \ref{fig3}, indicating $\chi_{la}(F_{2,k})=2k-1$ for even $k$.
%We give the different local antimagic labeling of $F_{2,2}\,(P_{6})$ with Arumugam et al. that shown in Fig. \ref{fig3}.

\begin{figure}[htbp]
\centering
\includegraphics[height=50mm]{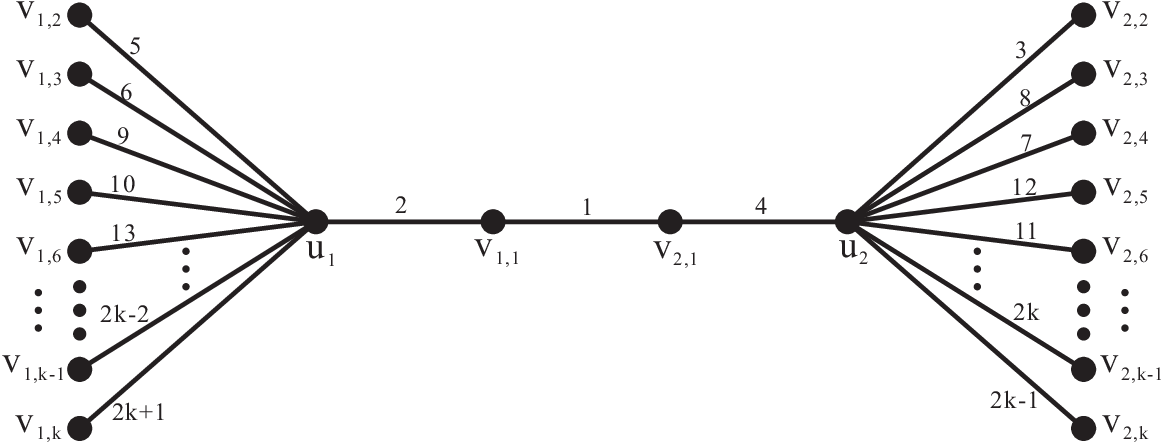}
\renewcommand{\figurename}{Fig.}
\caption{The local antimagic labeling of $F_{2,k}$.}\label{fig3}
\end{figure}

For $n=4$, the local antimagic labeling of $F_{4,k}$ is detailed as follows. Label the edges of the subgraph $F_{4,2}$ within $F_{4,k}$ by
\begin{equation*}
\begin{array}{llll}
f(u_{1}v_{1,1})=7, &f(u_{1}v_{1,2})=8,&f(u_{2}v_{2,1})=6, & f(u_{2}v_{2,2})=9,\\[5pt]
f(u_{3}v_{3,1})=4, &f(u_{3}v_{3,2})=11,&f(u_{4}v_{4,1})=10, &f(u_{4}v_{4,2})=5,\\[5pt]
f(v_{1,1}v_{2,1})=2, &f(v_{2,1}v_{3,1})=3,&f(v_{3,1}v_{4,1})=1. &
\end{array}
\end{equation*}
Then, label the remaining edges of $F_{4,k}$ by
\begin{equation*}
f(u_{i}v_{i,j})=\left\{
\begin{array}{ll}
4j-2+2i,   &\text{if }1\leq i\leq n,\,3\leq j\leq k~\text{and}~j~\text{odd},\\[3pt]
4j+5-2i, &\text{if }1\leq i\leq n,\,3\leq j\leq k~\text{and}~j~\text{even}.
\end{array}
\right.
\end{equation*}
Correspondingly, we find that $\omega(v_{1,1})=\omega(v_{2,2})=9$, $\omega(v_{2,1})=\omega(v_{4,1})=\omega(v_{3,2})=11$, $\omega(v_{3,1})=\omega(v_{1,2})=8$ and $\omega(v_{4,2})=5$, as well as $\omega(u_{i})=\frac{4k^{2}+7k}{2}$ and $12\leq\omega(v_{i,j})\leq 4k+3$ for $ 3\leq j\leq k$ when $1\leq i\leq n$. Thus,  $\chi_{la}(F_{4,k})=4k-3$ for even $k$.

For $n\geq 6$, we begin by labeling the edges $v_{i,1}v_{i+1,1}$ for $1\leq i\leq n-1$ with $[1, n-1]$ by
\begin{equation*}
f(v_{i,1}v_{i+1,1})=\left\{\begin{array}{ll}
\frac{n}{2}-\frac{i-1}{2},  &\text{if $1\leq i\leq n-1$~and~$i$~odd}, \\[3pt]
n-\frac{i}{2}, &\text{if $1\leq i\leq n-1$~and~$i$~even}.
\end{array}
\right.
\end{equation*}
Next, label the edges $u_{i}v_{i,j}$ for $1\leq i\leq n$ and $1\leq j\leq k$ with $[n,nk+n-1]$ by
\begin{equation*}
f(u_{i}v_{i,j})=\left\{\begin{array}{ll}
2n-2, &\text{if }i=1~\text{and}~j=1,\\
n-3+i,  &\text{if }3\leq i\leq n,~i~\text{odd~and}~j=1,\\
n-1+i,  &\text{if }2\leq i\leq n,~i~\text{even~and}~j=1,\\
2n+1,   &\text{if }i=1~\text{and}~j=2,\\
3n+2-i,  &\text{if }3\leq i\leq n,~i~\text{odd~and}~j=2,\\
3n-i,    &\text{if }2\leq i\leq n,~i~\text{even~and}~j=2,\\
jn-2+2i, &\text{if $1\leq i\leq n,\,3\leq j\leq k$~and~$j$~odd},\\
jn+n+1-i, &\text{if $1\leq i\leq n,\,3\leq j\leq k$~and~$j$~even}.
\end{array}\right.
\end{equation*}
From the above, we can deduce that $\omega(u_{i})=\frac{k^{2}n+2kn-k}{2}$ for $1\leq i\leq n$, $\omega(v_{i,1})=\frac{5n-4}{2}$ for odd $i$ and $\omega(v_{i,1})=\frac{5n}{2}$ for even $i$ when $1\leq i\leq n-1$, $\omega(v_{n,1})=2n$, as well as $2n\leq \omega(v_{i,j})\leq nk+n-1$ for $1\leq i\leq n$ and $2\leq j\leq k$.
It is easy to see that this local antimagic labeling $f$ uses $nk-n+1$ colors, and $\chi_{la}(F_{n,k})\leq nk-n+1$ for $n$ even. So,  $\chi_{la}(F_{n,k})=nk-n+1$ for $k$ even in this case.

In conclusion, the local antimagic chromatic number of firecracker graph $F_{n,k}$ is $\chi_{la}(F_{n,k})=nk-n+1$ for $n,k\geq 2$. This completes the proof.
\end{proof}

Obviously, a firecracker graph is a tree, so its proper vertex chromatic number is 2. Given the relationship between the local antimagic total chromatic number and the proper vertex chromatic number, we now present the local antimagic total chromatic number of firecracker graphs in the following theorems.

\begin{theorem}\label{Theorem 2.3}
For firecracker graph $F_{n,k}$, the lower bound of the local antimagic total chromatic number satisfies the following inequality
\begin{equation*}
  \chi_{lat}(F_{n,k})\geq \chi(F_{n,k})=2.
\end{equation*}
\end{theorem}

If $k=1$, let us discuss the local antimagic total chromatic number of $F_{n,1}$. When $n=1$, $F_{1,1}$ is the complete graph $K_2$, with a clear local antimagic total chromatic number 2, achieved by the local antimagic total labeling $g(u_1)=2$, $g(u_2)=3$ and $g(u_{1}u_{2})=1$. For $n=2$, $F_{2,1}$ is the path $P_4$ with the local antimagic total chromatic number 3, as detailed in \cite[Theorem 3.4]{Lau-7}. We then examine $\chi_{lat}(F_{n,1})$ for $n\geq 3$.% for $e=u_{1}u_{2}\in F_{1,1}$
\begin{theorem}
For $n\geq 3$, the local antimagic total chromatic number of firecracker graphs $F_{n,1}$ is
\begin{equation*}
  \chi_{lat}(F_{n,1})=2.
\end{equation*}
\end{theorem}

\begin{proof}
In accordance with Theorem \ref{Theorem 2.3}, the lower bound of $\chi_{lat}(F_{n,1})$ is established as $\chi_{lat}(F_{n,1})\geq \chi(F_{n,1})=2$. We now determine the upper bound of $\chi_{lat}(F_{n,1})$. Define a bijection  $g:\,V(F_{n,1})\cup E(F_{n,1}) \rightarrow[1, 4n-1]$, and consider the following two cases.

Case 1. $n$ is odd.

For $n=3$, the local antimagic total labeling is specified as follows:

$ g(u_{1}v_{1,1})=4, \hspace{0.8em}g(u_{2}v_{2,1})=6, \hspace{0.8em} g(u_{3}v_{3,1})=5, \hspace{0.8em} g(v_{1,1}v_{2,1})=1, \hspace{0.8em} g(v_{2,1}v_{3,1})=2,$

$g(u_{1})=8, \hspace{0.8em}g(u_{2})=10, \hspace{0.8em}g(u_{3})=7, \hspace{0.8em} g(v_{1,1})=11, \hspace{0.8em} g(v_{2,1})=3,\hspace{0.8em} g(v_{3,1})=9.$

It is straightforward to deduce that $\omega_{t}(u_1)=\omega_{t}(u_3)=\omega_{t}(v_{2,1})=12$ and $\omega_{t}(u_2)=\omega_{t}(v_{1,1})=\omega_{t}(v_{3,1})=16$. Thus $\chi_{lat}(F_{3,1})=2$.

For $n\geq 5$, label the edges and vertices of $F_{n,1}$ by
\begin{center}
$g(v_{i,1}v_{i+1,1})=\begin{cases}
\frac{n+i}{2}, & \text{if $1\leq i<n-1$ and odd $i$},\\[4pt]
\frac{i+1}{2}, & \text{if $1\leq i<n-1$ and even $i$},\\[4pt]
\frac{3n+1}{2}, &\text{if $i=n-1$}. \end{cases}$\vspace{0.4em}
$g(v_{i,1}v_{i+1,1})=\begin{cases}
n+1, & \text{if } i=1,\\[4pt]
\frac{3n+2-i}{2}, & \text{if $1< i<n$ and odd $i$},\\[4pt]
\frac{4n-i}{2}, & \text{if $1< i\leq n-5$ and even $i$},\\[4pt]
\frac{5n-2-i}{2}, & \text{if } i=n-3,n-1,\\[4pt]
n, &\text{if $i=n$}. \end{cases}$\vspace{0.4em}
 $g(v_{i,1})=\begin{cases}
\frac{7n-3}{2}, & \text{if } i=1,\\[4pt]
\frac{6n-1-i}{2}, & \text{if $1< i\leq n$ and odd $i$},\\[4pt]
\frac{5n-1-i}{2}, & \text{if $1< i\leq n-5$ and even $i$},\\[4pt]
\frac{3n+3}{2}, & \text{if } i=n-3,\\[4pt]
\frac{n-1}{2}, &\text{if $i=n-1$}. \end{cases}$\vspace{0.4em}
$g(u_{i})=\begin{cases}
n+1, & \text{if } i=1,\\[4pt]
\frac{7n-4+i}{2}, & \text{if $1< i<n$ and odd $i$},\\[4pt]
\frac{6n+i}{2}, & \text{if $1< i\leq n-5$ and even $i$},\\[4pt]
\frac{5n+1+i}{2}, & \text{if } i=n-3,n-1,\\[4pt]
4n-1, &\text{if $i=n$}. \end{cases}$
\end{center}

From the above labeling, it can be obtained that $\omega_{t}(v_{1,1})=\omega_{t}(v_{3,1})=\cdots=\omega_{t}(v_{n,1})=\omega_{t}(u_{2})=\omega_{t}(u_{4})=\cdots=\omega_{t}(u_{n-1})=5n$, and
$\omega_{t}(v_{2,1})=\omega_{t}(v_{4,1})=\cdots=\omega_{t}(v_{n-1,1})=\omega_{t}(u_{1})=\omega_{t}(u_{3})=\cdots=\omega_{t}(u_{n})=5n-1$.
There are only two different colors induced by the local antimagic total labeling of $F_{n,1}$, confirming that $\chi_{lat}(F_{n,1})=2$ for $n\geq 5$ and odd $n$.

Case 2. $n$ is even.

For $n=4$, there is a local antimagic total labeling of $F_{4,1}$ using 2 colors, that is,
\begin{equation*}
  \begin{array}{lllll}
  g(v_{1,1}v_{2,1})=1, &g(v_{2,1}v_{3,1})=3, &g(v_{3,1}v_{4,1})=2,& g(u_{1}v_{1,1})=4,  &g(u_{2}v_{2,1})=8, \\[3pt]
 g(u_{3}v_{3,1})=5, & g(u_{4}v_{4,1})=9,& g(v_{1,1})=15, & g(v_{2,1})=6,  & g(v_{3,1})=10,\\[3pt]
g(v_{4,1})=7,& g(u_{1})=14, &g(u_{2})=12, &g(u_{3})=13,&g(u_{4})=11.
\end{array}
\end{equation*}
Thus, $\omega_{t}(v_{1,1})=\omega_{t}(v_{3,1})=\omega_{t}(u_{2})=\omega_{t}(u_{4})=20$, and $\omega_{t}(v_{2,1})=\omega_{t}(v_{4,1})=\omega_{t}(u_{1})=\omega_{t}(u_{3})=18$. Consequently, $\chi_{lat}(F_{4,1})=2$.

For $n\geq 6$, the edges and vertices of $F_{n,1}$ are labeled similarly by

\begin{center}
$g(v_{i,1}v_{i+1,1})=\begin{cases}
\frac{i+1}{2}, & \text{if $1\leq i<n-1$ and odd $i$},\\[4pt]
 \frac{n+i}{2}, & \text{if $1\leq i<n-1$ and even $i$},\hspace{6em}\\[4pt]
\frac{3n}{2}, &\text{if $i=n-1$}. \end{cases}$\\[6pt]
$g(v_{i,1}v_{i+1,1})=\begin{cases}
n,    & \text{if } i=2,\\[8pt]
\frac{4n-1-i}{2}, & \text{if  $1\leq i<n-3$, odd $i$ and $i\neq n-5$},\\[4pt]
\frac{3n+2-i}{2}, & \text{if $2< i<n$ and even $i$},\\[4pt]
2n+2, & \text{if } i=n-5,\\[4pt]
2n, &\text{if $i=n-1$}. \end{cases}$\\[8pt]
$g(v_{i,1})=\begin{cases}
	3n-1, & \text{if } i=1,\\[4pt]
\frac{7n-4}{2}, & \text{if } i=2,\\[4pt]
	\frac{5n-1-i}{2}, & \text{if $1< i< n-1$ and odd $i$},\hspace{5em}\\[4pt]
	\frac{6n-2-i}{2}, & \text{if $2< i\leq n$ and even $i$},\\[4pt]
	\frac{3n+4}{2}, & \text{if } i=n-5,\\[4pt]
	\frac{n}{2}, &\text{if $i=n-1$}. \end{cases}$\\[8pt]
$g(u_{i})=\begin{cases}
		4n-1, & \text{if } i=2,\\[4pt]
		\frac{6n+1+i}{2}, & \text{if $1<i< n-1$, odd $i$ and $i\neq n-5$},\\[4pt]
		\frac{7n-4+i}{2}, & \text{if $2< i\leq n$ and even $i$},\\[4pt]
		3n-2, & \text{if } i=n-5,\\[4pt]
		3n, &\text{if $i=n-1$}. \end{cases}$
\end{center}
Consequently, $\omega_{t}(v_{1,1})=\omega_{t}(v_{3,1})=\cdots=\omega_{t}(v_{n-1,1})=\omega_{t}(u_{2})=\omega_{t}(u_{4})=\cdots=\omega_{t}(u_{n})=5n-1$, and
$\omega_{t}(v_{2,1})=\omega_{t}(v_{4,1})=\cdots=\omega_{t}(v_{n-1,1})=\omega_{t}(u_{1})=\omega_{t}(u_{3})=\cdots=\omega_{t}(u_{n})=5n$.  This is obviously a local antimagic total labeling that employs only two colors, and so $\chi_{lat}(F_{n,1})=2$ for $n\geq 6$ and even $n$.

To sum up, it follows that $\chi_{lat}(F_{n,1})=2$ for $n\geq 3$.\end{proof}

From the local antimagic total labeling of the firecracker graph $F_{n,1}$, we can know the local antimagic labeling of the join graph of $F_{n,1}$ and an empty graph $K_1$, denoted as $F_{n,1}\vee K_{1}$. Just need to assign the labels of the vertices from the local antimagic total labeling to the edges joining $K_1$ in the join graph $F_{n,1}\vee K_{1}$. Specifically, let $g(v_{i,1})=f(av_{i,1})$ and $g(u_{i})=f(au_{i})$, where $a\in V(K_{1})$ in  $F_{n,1}\vee K_{1}$, and $g$ is a local antimagic total labeling of $F_{n,1}$. Moreover, $f$, derived from $g$, becomes a local antimagic labeling of $F_{n,1}\vee K_{1}$. Thus the local antimagic chromatic number of graph $F_{n,1}\vee K_{1}$ can be determined.

\begin{corollary}
  For $n\geq 3$, $\chi_{la}(F_{n,1}\vee K_{1})=3$.
\end{corollary}

For $k\geq 2$, $F_{1,k}$ is a star, and thus its local antimagic total chromatic number is 2. $F_{2,2}$ is the path $P_{6}$, and the local antimagic total chromatic number of $F_{2,2}$ can be obtained from \cite[Theorem 3.4]{Lau-7}. Next, we will give the local antimagic total chromatic number of $F_{2,k}$ for $k\geq 3$.

\begin{theorem}
  For firecracker graph $F_{2,k}$, $k\geq 3$, the local antimagic total chromatic number of $F_{2,k}$ is bounded by
\begin{equation*}
2\leq \chi_{lat}(F_{2,k})\leq 3.
\end{equation*}
\end{theorem}

\begin{proof}
From Theorem \ref{Theorem 2.3}, the lower bound of $\chi_{lat}(F_{2,k})$ for $k\geq 3$ can be obtained,
that is, $\chi_{lat}(F_{2,k})\geq\chi(F_{2,k})=2$.

We now give part of the local antimagic total labeling for $F_{2,k}$ by

$g(v_{1,1}v_{2,1})=1$,  \hspace{1em}$ g(u_{i}v_{i,1})=2k-1+i$ for $i=1,2$,\vspace{0.4em}

$ g(u_{i}v_{i,j})=\begin{cases}
     2j-3+i, & \text{for $2\leq j\leq k$, even $j$ and $i=1,2$}, \\
     2j-i, & \text{for $2\leq j\leq k$, odd $j$ and $i=1,2$}. \end{cases}$\vspace{0.4em}

$g(v_{i,j})=\begin{cases}
     4k+8-2j-i, & \text{for $2\leq j\leq k$, even $j$ and $i=1,2$}, \\
     4k+5-2j+i, & \text{for $2\leq j\leq k$, odd $j$ and $i=1,2$}. \end{cases}$\vspace{0.6em}

\noindent Next, the labels about the remaining vertices $u_{i}$, $v_{i,1}$ in graph $F_{2,k}$ are determined based on the parity of $k$. For odd $k$, $g(u_{i})=2k+4-i$ for $i=1,2$, $g(v_{i,1})=2k+3+i$ for $i=1,2$. For even $k$, $g(u_{i})=2k+7-2i$ for $i=1,2$, $g(v_{i,1})=2k+6-2i$ for $i=1,2$. It is not difficult to check that, when $k$ is odd, $\omega_{t}(v_{1,1})=\omega_{t}(v_{i,j})=4k+5$ for $2\leq j\leq k$ and $i=1,2$, $\omega_{t}(v_{2,1})=4k+7$, $\omega_{t}(u_{1})=\omega_{t}(u_{2})=\frac{2k^{2}+7k+5}{2}$. When $k$ is even,  $\omega_{t}(v_{1,1})=\omega_{t}(v_{i,j})=4k+5$ for $2\leq j\leq k$ and $i=1,2$, $\omega_{t}(v_{2,1})=4k+4$, $\omega_{t}(u_{1})=\omega_{t}(u_{2})=\frac{2k^{2}+7k+8}{2}$. This local antimagic total labeling of $F_{2,k}$ induces three distinct colors, and so $\chi_{lat}(F_{2,k})\leq 3$ for $k\geq 3$.\end{proof}
%There are 3 different colors induced by this local antimagic total labeling of graph $F_{2,k}$,  so

\begin{theorem}
  For $n\geq 3$ and $k\geq 2$, the local antimagic total chromatic number of the firecracker graph $F_{n,k}$ is bounded as follows:
\begin{equation*}
  2\leq \chi_{lat}(F_{n,k})\leq 3.
\end{equation*}
\end{theorem}

\begin{proof}
From Theorem \ref{Theorem 2.3}, the lower bound of the local antimagic total chromatic number of $F_{n,k}$ for $n\geq 3$ and $k\geq 2$ is given by $\chi_{lat}(F_{n,k})\geq\chi(F_{n,k})=2$.

The upper bound of the local antimagic total chromatic number of $F_{n,k}$ is determined by the colors induced by the local antimagic total labeling. Define a bijection $g:V(F_{n,k})\cup E(F_{n,k})\rightarrow [1,2nk+2n-1]$, and let $U=[1,2nk+2n-1]$.
Discuss the local antimagic total labeling of graph $F_{n,k}$ through the following cases.

Case 1. $k$ is even.

Initially, we utilize the set $A=[2n,nk-1]\cup[nk+3n,2nk+n-1]$ to label the edges and vertices of $F_{n,k}$ excluding its subgraph $F_{n,2}$.

$g(u_{i}v_{i,j})=\begin{cases}
nj-n-1+i, & \text{if $1\leq i\leq n$, $3\leq j\leq k$ and odd $j$},\\[4pt]
nj-i,        & \text{if $1\leq i\leq n$, $3\leq j\leq k$ and even $j$}.\end{cases}$\vspace{0.4em}

$g(v_{i,j})=\begin{cases}
2nk+4n-nj-i, & \text{if $1\leq i\leq n$, $3\leq j\leq k$ and odd $j$},\\[4pt]
2nk+3n-1-nj+i,        & \text{if $1\leq i\leq n$, $3\leq j\leq k$ and even $j$}.\end{cases}$\vspace{0.6em}

Subsequently, employ the set $B=U\setminus A$ to label the subgraph $F_{n,2}$ of $F_{n,k}$, where $B=[1,2n-1]\cup[nk,nk+3n-1]\cup[2nk+n,2nk+2n-1]$.
It is further divided into the following two subcases according to the parity of $n$.

Subcase 1.1.  $n$ is odd.

When $n=3$, the local antimagic total labeling of $F_{3,k}$ for even $k$ is specified as follows:

$g(v_{1,1}v_{2,1})=1,\hspace{0.5em} g(v_{2,1}v_{3,1})=2, \hspace{0.5em}g(u_{1}v_{1,2})=3, \hspace{0.5em}g(u_{2}v_{2,2})=4, \hspace{0.5em}g(u_{3}v_{3,2})=5,$

$g(u_{1}v_{1,1})=3k+5,\hspace{0.5em}g(u_{2}v_{2,1})=3k+3,\hspace{0.5em}g(u_{3}v_{3,1})=3k+6,$\vspace{0.5em}

$\hspace{-0.5em}\begin{array}{lll}
g(u_{1})=3k+7,&g(u_{2})=3k+8,&g(u_{3})=3k+4,\\[4pt]
g(v_{1,1})=3k+2,&g(v_{2,1})=3k+1,&g(u_{3}v_{3,1})=3k,\\[4pt]
g(v_{1,2})=6k+5,&g(v_{2,2})=6k+4,&g(u_{3}v_{3,2})=6k+3,
\end{array}$\vspace{0.6em}

For $n\geq 5$, the corresponding local antimagic total labeling of $F_{n,k}$ for odd $n$ and even $k$ is obtained by\\

$g(v_{i,1}v_{i+1,1})=\begin{cases}
i, & \text{if $1\leq i<n$ and odd $i$},\\
n+1-i, &\text{if $1\leq i<n$ and even $i$}.\end{cases}$\vspace{0.5em}

$g(u_{i}v_{i,1})=\begin{cases}
nk+2n-1, & \text{if $ i=1$},\\
\frac{2nk+4n-3-i}{2},& \text{if $2\leq i< n$ and odd $i$},\\[4pt]
\frac{2nk+3n-1-i}{2}, &\text{if $2\leq i\leq n$ and even $i$},\\[4pt]
\frac{2nk+5n-3}{2}, & \text{if $i=n$}.\end{cases}$\vspace{0.5em}

$g(u_{i}v_{i,2})=\begin{cases}
n, & \text{if $ i=1$},\\
n+1+i, & \text{if $2\leq i< n$ and odd $i$},\\
n-1+i,  &\text{if $2\leq i\leq n$ and even $i$},\\
n+2, & \text{if $ i=n$}.\end{cases}$\vspace{0.5em}

$g(u_{i})=\begin{cases}
\frac{2nk+5n-1}{2}, & \text{if $ i=1$},\\[4pt]
\frac{2nk+5n-2-i}{2},& \text{if $2\leq i< n$ and odd $i$},\\[4pt]
\frac{2nk+6n-i}{2}, &\text{if $2\leq i\leq n$ and even $i$},\\[4pt]
nk+2n-2, & \text{if $ i=n$}.\end{cases}$\vspace{0.5em}

$g(v_{i,1})=\begin{cases}
nk+n-1, & \text{if $ i=1$},\\
\frac{2nk-3+i}{2}, & \text{if $2\leq i\leq n$ and odd $i$},\\[4pt]
\frac{2nk+n-3+i}{2}, &\text{if $2\leq i\leq n$ and even $i$}.\end{cases}$\vspace{0.5em}

$g(v_{i,2})=\begin{cases}
2nk+2n-1, & \text{if $ i=1$},\\
2nk+2n-2-i, & \text{if $2\leq i< n$ and odd $i$},\\[4pt]
2nk+2n-i,  &\text{if $2\leq i\leq n$ and even $i$},\\[4pt]
2nk+2n-3, & \text{if $ i=n$}.\end{cases}$\vspace{0.6em}

In this subcase, it can be calculated from the above labels that

$\omega_{t}(v_{s,1})=\omega_{t}(v_{i,j})=2nk+3n-1$ for $1\leq s,i\leq n$, $2\leq j\leq k$ and odd $s$,

$\omega(v_{s,1})=2nk+3n-2$ for $1\leq s\leq n$ and even $s$,

$\omega_{t}(u_{i})=\frac{nk^{2}+4nk+7n-k-1}{2}$ for $1\leq i\leq n$.

Clearly, this local antimagic total labeling of $F_{n,k}$ induces three different colors, and so $\chi_{lat}(F_{n,k})\leq 3$ for odd $n$ and even $k$.

Subcase 1.2.  $n$ is even.

The local antimagic total labeling for even $n$ is defined as follows:\\

\vspace{-0.8em}
$g(v_{i,1}v_{i+1,1})=\begin{cases}
                       \frac{i+1}{2}, & \text{if $1\leq i<n$ and odd $i$},\\[4pt]
                       \frac{n+i}{2}, & \text{if $1\leq i<n$ and even $i$}.\end{cases}$\vspace{0.5em}

$g(u_{i}v_{i,1})=\begin{cases}
nk+n-\frac{i+1}{2}, & \text{if $1\leq i\leq n-2$ and odd $i$},\\[4pt]
nk-1+\frac{n-i}{2}, & \text{if $1\leq i\leq n-2$ and even $i$},\\[4pt]
nk+n,     &\text{if }i=n-1,\\[4pt]
nk-1+\frac{n}{2}, &\text{if }i=n.\end{cases}$\vspace{0.4em}

$g(u_{i}v_{i,2})=\begin{cases}
\frac{3n-3-i}{2}, & \text{if $1\leq i\leq n-2$ and odd $i$},\\[4pt]
\frac{4n-2-i}{2}, & \text{if $1\leq i\leq n-2$ and even $i$},\\[4pt]
\frac{3n-2}{2},     &\text{if } i=n-1,\\[4pt]
2n-1, &\text{if }i=n.\end{cases}$\vspace{0.4em}

$g(u_{i})=\begin{cases}
nk+2n+1+i, & \text{if $1\leq i\leq n-2$},\\[4pt]
nk+n+1+i,        & \text{if $i=n-1,n$}.\end{cases}$\vspace{0.4em}

$g(v_{i,1})=\begin{cases}
nk+2n-1, & i=1,\\[4pt]
nk-1+\frac{3n-i}{2}, & \text{if $1<i\leq n-2$ and odd $i$},\\[4pt]
nk+2n-2-\frac{i}{2},        & \text{if $1< i\leq n-2$ and even $i$},\\[4pt]
nk+\frac{n}{2},  &\text{if } i=n-1,\\[4pt]
nk+2n-2, &\text{if }i=n.\end{cases}$\vspace{0.4em}

$g(v_{i,2})=\begin{cases}
2nk+\frac{3n+1+i}{2}, & \text{if $1\leq i\leq n-2$ and odd $i$},\\[4pt]
2nk+n+\frac{i}{2},        & \text{if $1\leq i\leq n-2$ and even $i$},\\[4pt]
2nk+\frac{3n}{2},         & \text{if } i=n-1,\\[4pt]
2nk+n, &\text{if }i=n.\end{cases}$\vspace{0.4em}

\noindent According to the above labels, we have

$\omega_{t}(v_{s,1})=\omega_{t}(v_{i,j})=2nk+3n-1$ for $1\leq s,i\leq n$, $2\leq j\leq k$ and odd $s$,

$\omega_{t}(v_{s,1})=2nk+3n-3$ for $1\leq s\leq n$ and even $s$,

$\omega_{t}(u_{i})=\frac{nk^{2}+4nk+5n-k}{2}+2nk+2n$ for $1\leq i\leq n$.

\noindent In this subcase, the local antimagic total labeling of $F_{n,k}$ uses three distinct colors, and thus $\chi_{lat}(F_{n,k})\leq 3$ for even $n$ and $k$.

Therefore, it is verified that $\chi_{lat}(F_{n,k})\leq 3$ for even $k$.

Case 2. $k$ is odd.

Similarly, use the set $C=[n,nk-1]\cup[nk+3n,2nk+2n-1]$ to label the edges and vertices of the firecracker graph $F_{n,k}$ excluding its subgraph $F_{n,1}$, in the following way:\\
\vspace{0.5em}

$g(u_{i}v_{i,j})=\begin{cases}
     nj-n-1+i, & \text{for $1\leq i\leq n$, $2\leq j\leq k$, and even $j$}, \\
     nj-i, & \text{for $1\leq i\leq n$, $2\leq j\leq k$, and odd $j$}. \end{cases}$\vspace{0.5em}

$ g(v_{i,j})=\begin{cases}
     2nk+4n-nj-i, & \text{for $1\leq i\leq n$, $2\leq j\leq k$, and even $j$}, \\
     2nk+3n-1-nj+i, & \text{for $1\leq i\leq n$, $2\leq j\leq k$, and odd $j$}. \end{cases}$\vspace{0.5em}

Then, use the set $D=U\setminus C=[nk,nk+3n-1]$ to label the subgraph $F_{n,1}$ of $F_{n,k}$. We discuss the parity of $n$ in a manner consistent with previous discussion.

Subcase 2.1. $n$ is odd.

For $n=3$,  present the local antimagic toal labeling of $F_{3,k}$ by the following formulas:

$g(v_{1,1}v_{2,1})=1,\hspace{1em}g(v_{2,1}v_{3,1})=2,\hspace{1em}g(u_{1}v_{1,1})=3k+7,\hspace{1em}g(u_{2}v_{2,1})=3k+8$

$g(u_{3}v_{3,1})=3k+5,\hspace{1em}g(v_{1,1})=3k,\hspace{1em}g(v_{2,1})=3k+2,\hspace{1em}g(v_{3,1})=3k+1$,

$g(u_{1})=3k+4,\hspace{1em}g(u_{2})=3k+3,\hspace{1em}g(u_{3})=3k+6$.

When $n\geq 5$, give the local antimagic toal labeling of $F_{n,k}$ by\vspace{0.6em}

$g(v_{i,1}v_{i+1,1})=\begin{cases}
\frac{2n-1-i}{2}, & \text{if $1\leq i\leq n$ and odd $i$},\\[4pt]
\frac{i}{2}, & \text{if $1\leq i\leq n$ and even $i$}. \end{cases}$\vspace{0.5em}

$g(u_{i}v_{i,1})=\begin{cases}
\frac{2nk+4n-1-i}{2}, & \text{if $1\leq i\leq n-1$ and odd $i$},\\[4pt]
\frac{2nk+6n-i}{2}, & \text{if $1\leq i\leq n-1$ and even $i$},\\[4pt]
\frac{2nk+5n-1}{2}, & \text{if $ i= n$}.\end{cases}$\vspace{0.5em}

$g(v_{i,1})=\begin{cases}
\frac{2nk+1+i}{2}, & \text{if $1\leq i\leq n-1$ and odd $i$},\\[4pt]
\frac{2nk+n-1+i}{2}, & \text{if $1\leq i\leq n-1$ and even $i$},\\[4pt]
nk, & \text{if $ i= n$}.\end{cases}$\vspace{0.5em}

$g(u_{i})=\begin{cases}
\frac{2nk+4n-1+i}{2}, & \text{if $1\leq i\leq n-1$ and odd $i$},\\[4pt]
\frac{2nk+2n-2+i}{2}, & \text{if $1\leq i\leq n-1$ and even $i$},\\[4pt]
\frac{2nk+3n-1}{2}, & \text{if $ i= n$}.\end{cases}$\vspace{0.6em}

\noindent By the above labels, we obtain that

$\omega_{t}(v_{s,1})=\omega_{t}(v_{i,j})=2nk+3n-1$ for $1\leq s,i\leq n$, $2\leq j\leq k$ and odd $s$,

$\omega_{t}(v_{s,1})=\frac{4nk+9n-1}{2}$ for $1\leq s\leq n$ and even $s$,

$\omega_{t}(u_{i})=\frac{nk^{2}+4nk+7n-k-1}{2}$ for $1\leq i\leq n$.

\noindent These are three distinct colors, and so $\chi_{lat}(F_{n,k})\leq 3$ when $n$ and $k$ are odd.

Subcase 2.2.  $n$ is even. \vspace{0.4em}

We give the following labeling\\

$g(v_{i,1}v_{i+1,1})=\begin{cases}
                       \frac{2n-1-i}{2}, & \text{if $1\leq i<n$ and odd $i$},\\[4pt]
                       \frac{i}{2}, & \text{if $1\leq i<n$ and even $i$}.\end{cases}$\vspace{0.4em}

$ g(u_{i}v_{i,1})=\begin{cases}
           \frac{2nk+3n-1-i}{2}, & \text{if $1\leq i\leq n$ and odd $i$},\\[4pt]
           \frac{2nk+4n-i}{2}, & \text{if $1\leq i\leq n$ and even $i$}.\end{cases}$\vspace{0.4em}

$g(u_{i})=\begin{cases}
    \frac{2nk+5n-1+i}{2}, & \text{if $1\leq i\leq n$ and odd $i$},\\[4pt]
    \frac{2nk+4n-2+i}{2}, & \text{if $1\leq i\leq n$ and even $i$}.\end{cases}$\vspace{0.4em}

$g(v_{i,1})=\begin{cases}
           \frac{2nk+n-3+i}{2}, & \text{if $1\leq i<n$ and odd $i$},\\[4pt]
            \frac{2nk-2+i}{2},    & \text{if $1\leq i<n$ and even $i$},\\[4pt]
                       nk+n-1, &\text{if } i=n.\end{cases}$\vspace{0.4em}

\noindent From the above labels in this subcase, it can be calculated that

$\omega_{t}(v_{s,1})=\omega_{t}(v_{i,j})=2nk+3n-1$ for $1\leq s,i\leq n$, $2\leq j\leq k$ and even $s$,

$\omega_{t}(v_{s,1})=2nk+3n-3$ for $1\leq s\leq n$ and odd $s$,

$\omega_{t}(u_{i})=\frac{nk^{2}+4nk+7n-k-1}{2}$ for $1\leq i\leq n$.

\noindent Obviously, there are three different colors used by the local antimagic total labeling, and so $\chi_{lat}(F_{n,k})\leq 3$ for even $n$ and odd $k$. In this case, it is concluded that $\chi_{lat}(F_{n,k})\leq 3$ for odd $k$. The conclusion is thus proved.\end{proof}

\subsection{The local antimagic chromatic number of edge corona-product graph $G\diamond H$}%$\chi_{la}()$

In this subsection, we determine the local antimagic chromatic number of the graph $G\diamond H$, and here $G$ is a star $S_{k}$ or a double star $S_{k_{1},k_{2}}$, while $H$ is the empty graph $\overline{K_{r}}$ or the complete graph $K_{2}$.

\subsubsection{$\chi_{la}(G\diamond \overline{K_{r}})$}

This subsection gives the local antimagic chromatic number of $G\diamond \overline{K_{r}}$, where $G$ is $S_k$ or $S_{k_{1},k_{2}}$.

For the edge-corona product of the star $S_{k}$ and the empty $\overline{K_{r}}$, let $V(S_k)=\{c,v_{1},v_{2},\cdots,v_{k}\}$ and $V_{1}=\{u_{i}^{j}|1\leq i\leq k,1\leq j\leq r\}$  represent the vertex sets of  $S_k$ and $n$ copies of $\overline{K_{r}}$, respectively. The vertex set and edge set of $S_{k}\diamond \overline{K_{r}}$ are given by $V(S_{k}\diamond \overline{K_{r}})=V(S_{k})\cup V_1$ and $E(S_{k}\diamond \overline{K_{r}})=\{cv_{i},cu_{i}^{j},v_{i}u_{i}^{j}|1\leq i\leq k,1\leq j\leq r\}$. The graph $S_{k}\diamond \overline{K_{r}}$ has order $kr+k+1$ and size $(2r+1)k.$

A double star, $S_{k_{1},k_{2}}\,(k_{1}\leq k_{2})$, is obtained by connecting an edge between the centers of two star graphs $S_{k_{1}}$ and $S_{k_{2}}$.
Let $V(S_{k_{1},k_{2}})=\{ c_{1},c_{2},v_{i}|1\leq i\leq k_{1}+k_{2}\}$ and $E(S_{k_{1},k_{2}})=\{c_{1}c_{2},c_{1}u_{i},c_{2}u_{s}|1\leq i\leq k_{1},k_{1}+1\leq s\leq k_{1}+k_{2}\}$ be the vertex set and edge set of the double star $S_{k_{1},k_{2}}$, respectively.
Consider the edge-corona product of a double star and an empty graph, that is, $S_{k_{1},k_{2}}\diamond \overline{K_{r}}$, whose vertex set and edge set is respectively $V(S_{k_{1},k_{2}}\diamond \overline{K_{r}})=V(S_{k_{1},k_{2}})\cup \{u_{i}^{j}|1\leq i\leq k_{1}+k_{2}+1,1\leq j\leq r\}$ and $E(S_{k_{1},k_{2}}\diamond \overline{K_{r}})=E(S_{k_{1},k_{2}})\cup \{c_{1}u_{1}^{j},c_{1}u_{i}^{j},v_{i-1}u_{i}^{j},c_{2}u_{1}^{j},c_{2}u_{s}^{j},v_{s-1}u_{s}^{j}|2\leq i\leq k_{1}+1,k_{1}+2\leq s\leq k_{1}+k_{2}+1,1\leq j\leq r\}$.

%{\it l.a.c.n.} of graphs $G\diamond O_{r}$, where $G$ is a family of trees, such as $S_n$, $S_{n_{1},n_{2}}$.%and $P_n$
First, we define the labeling matrix $\boldsymbol{\mathcal{M}}_{S_{k}}$ of the star $S_{k}$ as follows:
\begin{equation*}
	\boldsymbol{\mathcal{M}}_{S_{k}}={\small
		\begin{pmatrix}
			*& \boldsymbol{a}_{0} \\
			\boldsymbol{a}_{0}^{\rm{T}} & \bigstar
	\end{pmatrix}}
\end{equation*}
where $\boldsymbol{a}_{0}=(1,2,\cdots,k)$, and $\boldsymbol{a}_{0}^{\rm{T}}$ is the transpose of $\boldsymbol{a}_{0}$. The entries of the matrix $\boldsymbol{\mathcal{M}}_{S_{k}}$ are ordered by $c,v_{1},\ldots,v_{k}$. This matrix induces a local antimagic coloring labeling for $S_{k}$, utilizing $k+1$ colors, and hence $\chi_{la}(S_{k})=k+1$ by Lemma \ref{th}.

Let $\boldsymbol{\mathcal{M}}$ denote the labeling matrix of the graph $S_{k}\diamond \overline{K_{r}}$ with entries ordered as $c,v_{1},\ldots,v_{k},u_{1}^{1},\ldots,u_{k}^{1},u_{1}^{2},\ldots,u_{k}^{2},\ldots\ldots,u_{1}^{r},\ldots,u_{k}^{r}$. We partition $\boldsymbol{\mathcal{M}}$ in the following way,
\begin{equation*}
\boldsymbol{\mathcal{M}}=\begin{pmatrix}
\,\overline{\boldsymbol{\mathcal{M}}}\,\, \\  \,\underline{\boldsymbol{\mathcal{M}}}\,\,
\end{pmatrix}
=\begin{pmatrix}
  \boldsymbol{M}_{1}& \boldsymbol{M}_{2} \\
  \boldsymbol{M}_{2}^{\rm{T}} & \bigstar
\end{pmatrix},
\end{equation*}
where $\overline{\boldsymbol{\mathcal{M}}}=(\boldsymbol{M}_{1},\boldsymbol{M}_{2}) \in\mathbb{R}^{(k+1)\times (k+1+kr)}$ and $\underline{\boldsymbol{\mathcal{M}}}=(\boldsymbol{M}_{2}^{\rm{T}},\bigstar) \in\mathbb{R}^{kr\times (k+1+kr)}$ represent the upper and lower block matrices of $\boldsymbol{\mathcal{M}}$, respectively. $\boldsymbol{M}_{1}$ is a $(k+1)\times (k+1)$ square matrix, whose order of entries corresponding the vertices is $c,v_{1},v_{2},\cdots,v_{n}$, and $\boldsymbol{M}_{2}$ is a matrix of dimension $(k+1)\times kr$, whose row order of entries corresponding the vertices is $u_{1}^{1},u_{2}^{1},\cdots,u_{k}^{1},\cdots\,\cdots,u_{1}^{r},u_{2}^{r},\cdots,u_{k}^{r}$.
Notably, the upper block matrix $\overline{\boldsymbol{\mathcal{M}}}$ suffices for our discussion.

\begin{theorem}\label{th1}
For the star $S_{k}$ and empty graph $\overline{K_{r}}$, we have $\chi_{la}(S_{k}\diamond \overline{K_{r}})=3$.
\end{theorem}

\begin{proof}
The lower bound of $\chi_{la}(S_{k}\diamond \overline{K_{r}})$ is 3, as the proper vertex-coloring of the subgraph $K_{3}$ of $S_{k}\diamond \overline{K_{r}}$ requires 3 colors. Then, proceed with the following formulas:
\begin{equation*}
\begin{array}{ll}
\boldsymbol{a}_{j}=(2jk-k)\boldsymbol{e}+\boldsymbol{a}_{0},  & \text{for odd $j$ and $1\leq j\leq r$},\\[3pt]%&
\boldsymbol{a}_{j}=(2jk+1)\boldsymbol{e}-\boldsymbol{a}_{0},  & \text{for even $j$ and $1\leq j\leq r$}, \\[3pt]
\boldsymbol{A}_{j}=(2jk+k+1)\boldsymbol{I}-\boldsymbol{\Omega}, & \text{for odd $j$ and $1\leq j\leq r$},\\[3pt]
\boldsymbol{A}_{j}=2jk\boldsymbol{I}+\boldsymbol{\Omega}, & \text{for even $j$ and $1\leq j\leq r$},
\end{array}
\end{equation*}
where $\boldsymbol{a}_{0}=(1,2,\cdots,k)$ and $\boldsymbol{\Omega}=\textbf{diag}(1,2,\cdots,k)$, $\boldsymbol{e}$ denotes the vector of all 1's and $\boldsymbol{I}$ the identity matrix.

For odd $r$, the upper block matrix $\overline{\boldsymbol{\mathcal{M}}}_{1}$ of the labeling matrix $\boldsymbol{\mathcal{M}}_{1}$ for $S_{k}\diamond \overline{K_{r}}$ is given by
\begin{equation*}
\overline{\boldsymbol{\mathcal{M}}}_{1}=\begin{pmatrix}\boldsymbol{M}_{1} &\boldsymbol{M}_{2}\end{pmatrix}=
{\small{\left(\begin{array}{cc:ccccc}
 * & \boldsymbol{a}_{0} & \boldsymbol{a}_{1} & \boldsymbol{a}_{2} &\cdots & \boldsymbol{a}_{r-1} & \boldsymbol{a}_{r}\\[4pt]
 \boldsymbol{a}_{0}^{\rm{T}} &\bigstar & \boldsymbol{A}_{r} &\boldsymbol{A}_{r-1} &\cdots&\boldsymbol{A}_{2}&\boldsymbol{A}_{1}\\ \end{array}\right)}},
\end{equation*}
yielding the lower block matrix $\underline{\boldsymbol{\mathcal{M}}}_{1}=(\boldsymbol{M}_{2}^{\rm{T}},\bigstar)$.

Accordingly, the row sums of the matrix $\boldsymbol{\mathcal{M}}_{1}$ are
\begin{equation*}
\begin{array}{l}
(1)\,\text{for $\mathcal{R}_{1}$},\,\,\omega(c)=k^{2}r^{2}+\frac{(k^{2}+1)(r+1)}{2},\\[6pt]
(2)\,\text{from $\mathcal{R}_{2}$ to $\mathcal{R}_{k+1}$},\,\,\omega(v_{i})=kr^{2}+\frac{3kr+k+r+1}{2} \quad\quad\text{for $1\leq i\leq k$},\\[6pt]
(3)\,\text{from $\mathcal{R}_{k+2}$ to $\mathcal{R}_{kr+k+1}$},\,\,\omega(u_{i}^{j})=2rk+2k+1 \quad\text{for $1\leq i\leq k$ and $1\leq j\leq r$}.
\end{array}
\end{equation*}

For even $r$, certain transformations are required to construct the corresponding labeling matrix. Initially, for each $j\,(1\leq j\leq r)$, the vector $\boldsymbol{b}_{j}$ is obtained by reversing the elements of the vector $\boldsymbol{a}_{j}$. Specifically, if $\boldsymbol{a}_{j}=(a_{1},a_{2},\cdots,a_{k})$, then $\boldsymbol{b}_{j}=(a_{k},a_{k-1},\cdots,a_{1})$. Subsequently, make the following adjustments.
%the first entry in $\boldsymbol{a}_{j}$  is placed in the last position in $\boldsymbol{b}_{j}$, the second entry in the second-to-last position, and so on until the last entry in $\boldsymbol{a}_{j}$ is placed in the first position in $\boldsymbol{b}_{j}$.

Case 1. If $\frac{r}{2}$ is odd, adjust vector $\boldsymbol{b}_{\frac{r}{2}+1}=\big(rk+k+1,rk+k+2,\cdots,rk+2k-1,rk+2k\big)$ and the diagonal matrix $\boldsymbol{A}_{\frac{r}{2}}={\textbf{diag}}\big(rk+k,rk+k-1,\cdots,rk+2,rk+1\big)$ to $\boldsymbol{b}^{'}_{\frac{r}{2}+1}=\big(rk+2,rk+4,\cdots,rk+2k -2,rk+2k \big)$ and $\boldsymbol{A}^{'}_{\frac{r}{2}}={\textbf{diag}}\big(rk+2k-1,rk+2k-3,\cdots,rk+3,rk+1\big)$,  whose entries are increasing and decreasing by 2, respectively. So, the upper block matrix of the labeling matrix $\boldsymbol{\mathcal{M}}_{2'}$ for this case is
\begin{equation*}
\overline{\boldsymbol{\mathcal{M}}}_{2'}=\begin{pmatrix} \boldsymbol{M}_{1} &   \boldsymbol{M}_{2}^{'}\end{pmatrix}=
{\small{\left(\begin{array}{cc:cccccc}
 * & \boldsymbol{a}_{0} & \boldsymbol{b}_{1} &\cdots  &\boldsymbol{b}_{\frac{r}{2}} & \boldsymbol{b}'_{\frac{r+2}{2}} &\cdots &\boldsymbol{b}_{r}\\[6pt]
 \boldsymbol{a}_{0}^{\rm{T}} &\bigstar &\boldsymbol{A}_{r}   & \cdots & \boldsymbol{A}_{\frac{r+2}{2}} & \boldsymbol{A}'_{\frac{r}{2}} & \cdots & \boldsymbol{A}_{1}\end{array}\right)}},
\end{equation*}
and we obtain the lower block matrix $\underline{\boldsymbol{\mathcal{M}}}_{2'}=(\boldsymbol{M}_{2}^{'\,\rm{T}},\bigstar)$.

Case 2. If $\frac{r}{2}$ is even, analogous adjustments are made to the vector $\boldsymbol{b}_{\frac{r}{2}+1}=\big(rk+2k,\cdots,rk+k+1\big)$ and the diagonal matrix $\boldsymbol{A}_{\frac{r}{2}}={\textbf{diag}}\big(rk+1,\cdots,rk+k\big)$ to $\boldsymbol{b}^{'}_{\frac{r}{2}+1}$ and $\boldsymbol{A}^{'}_{\frac{r}{2}}$, respectively. The vector $\boldsymbol{b}^{'}_{\frac{r}{2}+2}$ is reversed to return to $\boldsymbol{a}_{\frac{r}{2}+2}$, and the diagonal entries of $\boldsymbol{A}_{\frac{r}{2}-1}$ are reordered to form $\boldsymbol{A}^{'}_{\frac{r}{2}-1}$. Thus, the upper block matrix of the labeling matrix $\boldsymbol{\mathcal{M}}_{2''}$ for this case is
%Again, the same reversal operation is performed on the vector $\boldsymbol{b}^{'}_{\frac{r}{2}+2}$, which returns to its original vector $\boldsymbol{a}_{\frac{r}{2}+2}$. The reordering of the diagonal entries within a diagonal matrix $\boldsymbol{A}_{\frac{r}{2}-1}$ gives rise to a new diagonal matrix $\boldsymbol{A}^{'}_{\frac{r}{2}-1}$, where the entries along the diagonal are arranged in a reversed order. Thus the upper block matrix of the labeling matrix $\boldsymbol{\mathcal{M}}_{2''}$ for this case is as follows,
\begin{equation*}
\overline{\boldsymbol{\mathcal{M}}}_{2''}=\begin{pmatrix} \boldsymbol{M}_{1} &   \boldsymbol{M}_{2}^{''}\end{pmatrix}=
{\small{\left(\begin{array}{cc:cccccc}
 * & \boldsymbol{a}_{0} & \boldsymbol{b}_{1} &\cdots  &\boldsymbol{b}'_{\frac{r+2}{2}} & \boldsymbol{b}'_{\frac{r+4}{2}} &\cdots &\boldsymbol{b}_{r}\\[6pt]
 \boldsymbol{a}_{0}^{\rm{T}} &\bigstar &\boldsymbol{A}_{r}   & \cdots & \boldsymbol{A}_{\frac{r}{2}} & \boldsymbol{A}'_{\frac{r-2}{2}} & \cdots & \boldsymbol{A}_{1}\end{array}\right)}},
\end{equation*}

The corresponding row sums for even $r$ are calculated as
\begin{equation*}
\begin{array}{l}
(1)\,\text{for $\mathcal{R}_{1}$},\,\,\omega(c)=k^{2}r^{2}+k+\frac{(k^{2}+k)r}{2},\\[6pt]
(2)\,\text{from $\mathcal{R}_{2}$ to $\mathcal{R}_{k+1}$},\,\,\omega(v_{i})=k(r^{2}+1)+\frac{3kr+r}{2}\quad\quad\text{for $1\leq i\leq k$},\\[6pt]
(3)\,\text{from $\mathcal{R}_{k+2}$ to $\mathcal{R}_{kr+k+1}$},\,\,\omega(u_{i}^{j})=2rk+2k+1\quad\text{for $1\leq i\leq k$ and $1\leq j\leq r$}.
\end{array}
\end{equation*}

To sum up, the labeling matrices derived above, based on their row sums, confirm a local antimagic labeling for the graph $S_{k}\diamond \overline{K_{r}}$, satisfying $\chi_{la}(S_{k}\diamond \overline{K_{r}})\leq 3$. This completes the proof for $\chi_{la}(S_{k}\diamond \overline{K_{r}})=3$.
\end{proof}

\begin{example}
   For star $S_{3}$ and empty graph $\overline{K_{6}}$,  $\chi_{la}(S_{3}\diamond \overline{K_{6}})=3$.
\end{example}

We establish that $\chi_{la}(S_{3}\diamond \overline{K_{6}})\geq 3$, and proceed to construct the upper block matrix of  the labeling matrix $\boldsymbol{\mathcal{M}}=(\overline{\boldsymbol{\mathcal{M}}},\underline{\boldsymbol{\mathcal{M}}})^{\rm{T}}$ for $S_{3}\diamond \overline{K_{6}}$ by
\begin{equation*}
\begin{aligned}
\overline{\boldsymbol{\mathcal{M}}}&=\begin{pmatrix}\boldsymbol{M}_{1} &\boldsymbol{M}_{2}\end{pmatrix}\\
&={\tiny{\left(\begin{array}{cccc:ccc ccc ccc ccc ccc ccc}
* &1 &2 & 3 & 6 & 5 & 4 & 10 & 11 & 12 & 18 & 17 & 16 & 20 & 22 & 24 & 30 &29 & 28& 34 & 35 & 36 \\[2pt]
1 & * & * & * &37 & * & *&33& * & * & 25 & * & * & 23& * &* & 13 & * & * & 9 & *& *\\
2 & * & * & * & * &38 & *& *&32 & * & * & 26 & * & *& 21 & * & *&14 & * & * & 8& *\\
3 & * & * & * & * & * &39& *& * & 31 & *  & * & 27 & *& * & 19 & * & *&15 & * & *& 7
\end{array}\right)}}
\end{aligned}
\end{equation*}
and the corresponding lower block matrix is $\underline{\boldsymbol{\mathcal{M}}}=(\boldsymbol{M}_{2}^{\rm{T}},\bigstar)$. From the labeling matrix $\boldsymbol{\mathcal{M}}$, we have
\begin{equation*}
\begin{array}{l}
\omega(c)=363\text{ for $\mathcal{R}_{1}$},\quad\omega(v_{1})=\omega(v_{2})=\omega(v_{3})=141 \text{ from $\mathcal{R}_{2}$ to $\mathcal{R}_{4}$},\\[6pt]
\omega(u_{i}^{j})=43\,\,\text{from $\mathcal{R}_{5}$ to $\mathcal{R}_{22}$}\,\,\text{for $1\leq i\leq 3$ and $1\leq j\leq 6$}.
\end{array}
\end{equation*}
The labeling matrix induces a local antimagic labeling of $S_{3}\diamond \overline{K_{6}}$ using 3 colors. So, $\chi_{la}(S_{3}\diamond \overline{K_{6}})\geq 3$.

For the double star $S_{k_{1},k_{2}}$, suppose its labeling matrix is $\boldsymbol{\mathcal{M}}\in \mathbb{R}^{t\times t}$, $t=(k_{1}+k_{2}+1)(r+1)+1$, whose entries of vertices are ordered as $c_{1},c_{2},v_{1},\cdots,v_{k_{1}},v_{k_{1}+1},\cdots,v_{k_{1}+k_{2}},u_{1}^{1}$, $u_{2}^{1},\cdots,u_{k_{1}+k_{2}+1}^{1},\cdots,u_{1}^{r},u_{2}^{r},\cdots,u_{k_{1}+k_{2}+1}^{r}$. Let us partition this matrix by the similar way as follows:
\begin{equation}
\boldsymbol{\mathcal{M}}=\left(\begin{array}{cc}
\overline{\boldsymbol{\mathcal{M}}}\\\underline{\boldsymbol{\mathcal{M}}}
\end{array}\right)
=\left(\begin{array}{cc}
\boldsymbol{M}_{1} & \boldsymbol{M}_{2}\\
\boldsymbol{M}_{2}^{{\rm T}} &\bigstar
\end{array}\right),
\end{equation}
where $\overline{\boldsymbol{\mathcal{M}}}\in \mathbb{R}^{(k_{1}+k_{2}+2)\times t}$ and $\underline{\boldsymbol{\mathcal{M}}}\in \mathbb{R}^{(k_{1}r+k_{2}r+r)\times t}$ represent the upper and lower block matrices of $\boldsymbol{\mathcal{M}}$, respectively. Here,
$\boldsymbol{M}_{1}\in \mathbb{R}^{(k_{1}+k_{2}+2)\times (k_{1}+k_{2}+2)}$, $\boldsymbol{M}_{2}\in \mathbb{R}^{(k_{1}+k_{2}+2)\times(k_{1}r+k_{2}r+r)}$. We only need the upper block matrix $\overline{\boldsymbol{\mathcal{M}}}$.

\begin{theorem}
For the double star $S_{k_{1},k_{2}}$, the local antiamgic chromatic number satisfies $3\leq\chi_{la}(S_{k_{1},k_{2}}\diamond \overline{K_{r}})\leq 4$.
\end{theorem}

\begin{proof}
The lower bound of $\chi_{la}(S_{k_{1},k_{2}}\diamond \overline{K_{r}})$ is 3, since the proper vertex-coloring of the subgraph $K_{3}$ of $S_{k_{1},k_{2}}\diamond \overline{K_{r}}$ requires at least 3 colors. We aim to achieve a local antimagic labeling of $S_{k_{1},k_{2}}\diamond \overline{K_{r}}$ from its labeling matrix using as fewer colors as possible. Start with the following formulas
\begin{equation*}
  \begin{array}{ll}
  \boldsymbol{b}_{p}=(2p-1)(k_{1}+k_{2}+1)\boldsymbol{e}+\boldsymbol{b}_{0}, &\text{for odd $p$ and $1\leq p\leq r$},\\[8pt]
  \boldsymbol{b}_{p}=(2p(k_{1}+k_{2}+1)+1)\boldsymbol{e}-\boldsymbol{b}_{0}, &\text{for even $p$ and $1\leq p\leq r$}.\\[8pt]
  \boldsymbol{B}_{p}=[(2p+1)(k_{1}+k_{2}+1)+1]\boldsymbol{I}-\boldsymbol{\Omega}_{0}, &\text{for odd $p$ and $1\leq p\leq r$}.\\[8pt]
  \boldsymbol{B}_{p}=2p(k_{1}+k_{2}+1)\boldsymbol{I}+\boldsymbol{\Omega}_{0}, &\text{for even $p$ and $1\leq p\leq r$}.
  \end{array}
\end{equation*}
where $\boldsymbol{b}_{0}=(1,2,\cdots,k_{1}+k_{2}+1)$ and $\boldsymbol{\Omega}_{0}={\textbf{diag}}(1,2,\cdots,k_{1}+k_{2}+1)$. Each vector and matrix is divided into three parts, such as $\boldsymbol{b}_{p}=(b_{p}^{1},\boldsymbol{b}_{p}^{2},\boldsymbol{b}_{p}^{3})$. Specifically, $b_{p}^{1}$ is the first component of $\boldsymbol{b}_{p}$, $\boldsymbol{b}_{p}^{2}$ and $\boldsymbol{b}_{p}^{3}$ includes the middle $k_{1}$ and last $k_{2}$ components of $\boldsymbol{b}_{p}$. Similarly, divide each $\boldsymbol{B}_{p}$ into the single element ${B}_{p}^{1}$ and two diagonal matrices $\boldsymbol{B}_{p}^{2}$, $\boldsymbol{B}_{p}^{3}$, corresponding respectively to the middle $k_{1}$ and last $k_{2}$ diagonal elements of the original diagonal matrix.

For odd $r$, the upper block matrix $\overline{\boldsymbol{\mathcal{M}}_{1}}$ of the labeling matrix $\boldsymbol{\mathcal{M}}_{1}$ for the graph $S_{k_{1},k_{2}}\diamond \overline{K_{r}}$ is

\begin{equation*}
\overline{\boldsymbol{\mathcal{M}}_{1}}=\left(\begin{array}{cccc:ccccccc}
* & b_{0}^{1} &\boldsymbol{b}_{0}^{2} &\bigstar & b_{1}^{1} &\boldsymbol{b}_{1}^{2} &\bigstar &\cdots\cdots & b_{r}^{1} &\boldsymbol{b}_{r}^{2} &\bigstar\\[5pt]
b_{0}^{1} & * & \bigstar & \boldsymbol{b}_{0}^{3} & B_{r}^{1} &\bigstar &\boldsymbol{b}_{1}^{3} &\cdots\cdots & B_{1}^{1} &\bigstar &\boldsymbol{b}_{r}^{3}\\[5pt]
\boldsymbol{b}_{0}^{2\,\rm{T}} & \bigstar & \bigstar& \bigstar &  \bigstar &\boldsymbol{B}_{r}^{2} &\bigstar &\cdots\cdots & \bigstar &\boldsymbol{B}_{1}^{2} &\bigstar\\[5pt]
\bigstar & \boldsymbol{b}_{0}^{3\,\rm{T}} & \bigstar & \bigstar & \bigstar & \bigstar &\boldsymbol{B}_{r}^{3} &\cdots\cdots &\bigstar& \bigstar &\boldsymbol{B}_{1}^{3}
\end{array}\right).
%=\left(\begin{array}{cc}\boldsymbol{M}_{1}^{1}&\boldsymbol{M}_{2}^{1}\end{array}\right)
\end{equation*}
The corresponding row sums of $\boldsymbol{\mathcal{M}}_{1}$ in this case are calculated as
\begin{equation*}
\begin{array}{l}
(1)\,\text{for $\mathcal{R}_{1}$},\,\,\omega(c_{1})=(k_{1}+1)[r^{2}(k_{1}+k_{2}+1)+k+1+\frac{rk_{1}+rk_{2}+k_{1}-k_{2}}{2}],\\[6pt]
(2)\,\text{for $\mathcal{R}_{2}$},\,\,\omega(c_{2})=\frac{k_{2}[(2r^{2}+r+1)(k_{1}+k_{2}+1)+2k_{1}+3+r]}{2}+\frac{(2r^{2}+3r+1)(k_{1}+k_{2}+1)+r+1}{2},\\[6pt]
(3)\,\text{from $\mathcal{R}_{3}$ to $\mathcal{R}_{k_{1}+k_{2}+2}$},\quad\omega(v_{i})=\frac{(2r^{2}+3r+1)(k_{1}+k_{2}+1)+r+1}{2}\text{ for $1\leq i\leq k_{1}+k_{2}$},\\[6pt]
(4)\,\text{from $\mathcal{R}_{k_{1}+k_{2}+3}$ to $\mathcal{R}_{(k_{1}+k_{2}+1)(r+1)+1}$},\quad\omega(u_{i}^{j})=(2r+2)(k_{1}+k_{2}+1)+1 \,\,\text{for}\\[6pt]
\text{$1\leq i\leq k_{1}+k_{2}+1$ and $1\leq j\leq r$}.
\end{array}
\end{equation*}

For even $r$, the labeling matrix is similarly discussed based on the parity of $\frac{r}{2}$, following the proof of Theorem \ref{th1}. The labeling matrices $\boldsymbol{\mathcal{M}}_{2'}$ and $\boldsymbol{\mathcal{M}}_{2''}$ are obtained by applying the same transformation and partitions as described above. When $r$ is even, the row sums of the labeling matrices $\boldsymbol{\mathcal{M}}_{2'}$ and $\boldsymbol{\mathcal{M}}_{2''}$ are%$\frac{r}{2}$
\begin{equation*}
\begin{array}{l}
(1)\,\text{for $\mathcal{R}_{1}$},\,\,\omega(c_{1})=\frac{(k_{1}+1)[(2r^{2}+r)(k_{1}+k_{2}+1)-2k_{2}+r+2]}{2},\\[6pt]
(2)\,\text{for $\mathcal{R}_{2}$},\,\,\omega(c_{2})=k_{2}(k_{1}+2)+\frac{[k_{2}(2r^{2}+r)+(2r^{2}+3r+2)](k_{1}+k_{2}+1)+(k_{2}+1)r}{2},\\[6pt]
(2)\,\text{from $\mathcal{R}_{3}$ to $\mathcal{R}_{k_{1}+k_{2}+2}$},\,\,\omega(v_{i})=\frac{(2r^{2}+3r+2)(k_{1}+k_{2}+1)+r}{2}\text{ for $1\leq i\leq k_{1}+k_{2}$},\hspace{2em}\\[6pt]
(3)\,\text{from $\mathcal{R}_{k_{1}+k_{2}+3}$ to $\mathcal{R}_{(k_{1}+k_{2}+1)(r+1)+1}$},\,\,\omega(u_{i}^{j})=(2r+2)(k_{1}+k_{2}+1)+1,\\[6pt]
\quad\,\text{ for $1\leq i\leq k_{1}+k_{2}+1$ and $1\leq j\leq r$}.
\end{array}
\end{equation*}

In conclusion, the derived labeling matrices for $S_{k_{1},k_{2}}\diamond \overline{K_{r}}$ facilitate a local antimagic labeling using four distinct colors, which are derived from the sums of the corresponding row entries. This demonstrates that $\chi_{la}(S_{k_{1},k_{2}}\diamond \overline{K_{r}})\leq 4$. Combining this with the established lower bound 3, we affirm that $3\leq \chi_{la}(S_{k_{1},k_{2}}\diamond \overline{K_{r}})\leq 4$. Thus the proof is complete.
\end{proof}

\begin{example}\label{E2}
{For $k_{1}=4,k_{2}=5,r=4$, we have $3\leq\chi_{la}(S_{4,5}\diamond \overline{K_{4}})\leq 4$.}
\end{example}
The labeling matrix $\boldsymbol{\mathcal{M}}_{S_{4,5}\diamond \overline{K_{4}}}$ is given by
\begin{equation*}
\boldsymbol{\mathcal{M}}_{S_{4,5}\diamond \overline{K_{4}}}=
\begin{pmatrix}
	\boldsymbol{M}_{1} & \boldsymbol{M}_{2}\\
	\boldsymbol{M}_{2}^{\rm{T}} &\bigstar
\end{pmatrix},
\end{equation*}
where the details of the matrices $\boldsymbol{M}_{1},\boldsymbol{M}_{2}$ are provided in Appendix \ref{A}. Compute the row sums of $\boldsymbol{\mathcal{M}}_{S_{4,5}\diamond \overline{K_{4}}}$, we get
\begin{equation*}
\begin{array}{l}
\quad\omega(c_{1})=890 \text{ for $\mathcal{R}_{1}$},\quad\omega(c_{2})=1172 \text{ for $\mathcal{R}_{2}$},\\[4pt]
\quad\omega(v_{i})=232 \text{ from $\mathcal{R}_{3}$ to $\mathcal{R}_{11}$ for $1\leq i\leq 9$}, \\[4pt]
\quad\omega(u_{i}^{j})=101 \text{ from $\mathcal{R}_{12}$ to $\mathcal{R}_{51}$ for $1\leq i\leq 10$ and $1\leq j\leq 4$}.\quad\quad\quad\quad\quad\quad\quad\quad\quad\quad\quad\quad\quad\quad\quad
\end{array}
\end{equation*}
Thus, a local antimagic labeling of $S_{4,5}\diamond \overline{K_{4}}$ is obtained using four distinct colors. We confirm that $3\leq\chi_{la}(S_{4,5}\diamond \overline{K_{4}})\leq 4$.

\subsubsection{$\chi_{la}(G\diamond rK_{2})$}%The Local Antimagic chromatic number of graph $G\diamond mK_{2}$

In this subsection, we explore the local antimagic chromatic number of the edge-corona product of a graph $G$ and $rK_{2}$. Let us start by establishing the lower bound for an arbitrary graph $G$.
Given a graph $G$, the proper vertex-coloring of $G\diamond rK_{2}$ uses at least 4 colors. This leads us to the following inequality for the local antimagic chromatic number
$$\chi_{la}(G\diamond rK_{2})\geq \chi(G\diamond rK_{2})\geq4.$$
For the star $S_{k}$, denote the vertex set and edge set of $S_{k}\diamond rK_{2}$ as
\begin{equation*}
  \begin{array}{c}%\label{}
      V(S_{k}\diamond rK_{2})=V(S_{k})\cup \{u_{i}^{j1},u_{i}^{j2}|1\leq i\leq k,1\leq j\leq r\}, \\[4pt]
      E(S_{k}\diamond rK_{2})=E(S_{k})\cup \{cu_{i}^{j1},cu_{i}^{j2},v_{i}u_{i}^{j1},v_{i}u_{i}^{j2},u_{i}^{j1}u_{i}^{j2}|1\leq i\leq k,1\leq j\leq r\}.
  \end{array}
\end{equation*}
Then, determine the value of $\chi_{la}(S_{k}\diamond rK_{2})$.

\begin{theorem}\label{The-31}
For the star $S_{k}$, we have $\chi_{la}(S_{k}\diamond rK_{2})=4$.
\end{theorem}

\begin{proof}
The lower bound of $\chi_{la}(S_{k}\diamond rK_{2})$ is 4, as the induced subgraph of $S_{k}\diamond rK_{2}$ contains $K_4$, implying $\chi_{la}(S_{k}\diamond rK_{2})\geq 4$. Then the upper bound of $\chi_{la}(S_{k}\diamond rK_{2})$ is obtained by the local antimagic labelings of $S_{k}\diamond rK_{2}$ constructed with the labeling matrices.

Let us start with the following calculations:
\begin{equation}\label{sk2-1}
\begin{array}{ll}
\boldsymbol{a}_{0}=(1,2,\cdots,k),\\
\boldsymbol{a}_{j}=jk\boldsymbol{e}+\boldsymbol{a}_{0}, & \text{for $1\leq j\leq r$ and $j$ odd},\\[4pt]
\boldsymbol{a}_{j}=(jk+k+1)\boldsymbol{e}-\boldsymbol{a}_{0}, & \text{for $1\leq j\leq r$ and $j$ even},\\[6pt]
\boldsymbol{A}_{0}=\mathrm{\textbf{diag}}(2k-1,2k-3,\cdots,1)\in \mathbb{R}^{k\times k},\\[4pt]
\boldsymbol{A}_{j}=(r+2j-1)k\boldsymbol{I}+\boldsymbol{A}_{0}, &\text{for $1\leq j\leq r$ and $j$ odd},\\[4pt]
\boldsymbol{A}_{j}=(r+2j+1)k\boldsymbol{I}-\boldsymbol{A}_{0}, & \text{for $1\leq j\leq r$ and $j$ even}.
\end{array}
\end{equation}

Vectorizing $\boldsymbol{A}_{0}$ into $\boldsymbol{b}_{0}$ involves stretching the diagonal elements of $\boldsymbol{A}_{0}$ to form the vector $\boldsymbol{b}_{0}$. It follows that
\begin{equation}\label{sk2-2}
\begin{array}{ll}
\boldsymbol{b}_{0}=(2k-1,2k-3,\cdots,1),\\[2pt]
\boldsymbol{b}_{j}=(rk+2jk-k+1)\boldsymbol{e}+\boldsymbol{b}_{0}, &\text{for $1\leq j\leq r$ and $j$ odd},\\[4pt]
\boldsymbol{b}_{j}=(rk+2jk+k+1)\boldsymbol{e}-\boldsymbol{b}_{0}, & \text{for $1\leq j\leq r$ and $j$ even}.\\[4pt]
\boldsymbol{\Omega}_{0}=\mathrm{\textbf{diag}}(1,2,\cdots,k)\in \mathbb{R}^{k\times k},\\[2pt]
\boldsymbol{B}_{j}=(3rk+jk)\boldsymbol{I}+\boldsymbol{\Omega}_{0}, & \text{for $1\leq j\leq r$ and $j$ odd},\\[4pt]
\boldsymbol{B}_{j}=(3rk+jk+k+1)\boldsymbol{I}-\boldsymbol{\Omega}_{0}, & \text{for $1\leq j\leq r$ and $j$ even}.\\[4pt]
\boldsymbol{B}_{r+j}=rk\boldsymbol{I}+\boldsymbol{B}_{j}, &\text{for $1\leq j\leq r$}.
\end{array}
\end{equation}

Denote the labeling matrix of $S_{k}\diamond rK_{2}$ as $\boldsymbol{\mathcal{M}}^{srk}$, whose entries are ordered as $c,v_{1}$, $\cdots,v_{k},\,u_{1}^{11},\cdots,u_{k}^{11},\cdots,u_{1}^{r1},\cdots,u_{k}^{r1},u_{1}^{12},\cdots,u_{k}^{12},\cdots,$ $u_{1}^{r2},\cdots,u_{k}^{r2}$. First, we give the labeling matrix for $S_{k}\diamond rK_{2}$ when $r$ is odd as follows:
\begin{equation*}
\boldsymbol{\mathcal{M}}_{1}^{srk}={\small\left(\begin{array}{cc:cccc:cccc}
* & \boldsymbol{a}_{0} & \boldsymbol{a}_{1}& \boldsymbol{a}_{2} & \cdots & \boldsymbol{a}_{r}& \boldsymbol{b}_{r}& \boldsymbol{b}_{r-1} & \cdots & \boldsymbol{b}_{1}\\[4pt]
\boldsymbol{a}_{0}^{\rm{T}}& \bigstar & \boldsymbol{A}_{r}& \boldsymbol{A}_{r-1}&\cdots &\boldsymbol{A}_{1} & \boldsymbol{B}_{1}& \boldsymbol{B}_{2}&\cdots &\boldsymbol{B}_{r}\\[4pt]
\hdashline
&&&&&&\\[-12pt]
\boldsymbol{a}_{1}^{\rm{T}}& \boldsymbol{A}_{r} &\bigstar & & & & \boldsymbol{B}_{r+1}& \bigstar & \cdots & \bigstar\\[4pt]
\boldsymbol{a}_{2}^{\rm{T}} &\boldsymbol{A}_{r-1} & &\bigstar & &  & \bigstar & \boldsymbol{B}_{r+2}& \cdots & \bigstar\\[4pt]
\vdots & \vdots & & & \ddots& & \vdots & \vdots &\ddots & \vdots \\
\boldsymbol{a}_{r}^{\rm{T}}&\boldsymbol{A}_{1} & & & & \bigstar& \bigstar & \bigstar& \cdots& \boldsymbol{B}_{2r}\\[2pt]\hdashline
&&&&&&\\[-12pt]
\boldsymbol{b}_{r}^{\rm{T}}& \boldsymbol{B}_{1} & \boldsymbol{B}_{r+1}& \bigstar & \cdots & \bigstar &\bigstar\\[4pt]
\boldsymbol{b}_{r-1}^{\rm{T}} &\boldsymbol{B}_{2}&\bigstar  & \boldsymbol{B}_{r+2}& \cdots & \bigstar & &\bigstar\\[4pt]
\vdots & \vdots & \vdots & \vdots &\ddots& \vdots & & &\ddots\\
\boldsymbol{b}_{1}^{\rm{T}} &\boldsymbol{B}_{r} & \bigstar & \bigstar& \cdots& \boldsymbol{B}_{2r} & & & &\bigstar
\end{array}\right)},
\end{equation*}

The colors are induced from $\boldsymbol{\mathcal{M}}_{1}^{srk}$ by calculating the row sums as follows:
\begin{equation*}
\begin{array}{l}
(1)\,\text{for $\mathcal{R}_{1}$},\,\,\omega(c_{1})=rk^{2}+4k^{2}+2k+\frac{3k^{2}(r^{2}+2r-3)+(kr-k)(2rk-2k+3)}{2},\\[6pt]
(2)\,\text{from $\mathcal{R}_{2}$ to $\mathcal{R}_{k+1}$},\,\,\omega(v_{i})=2rk+2k+\frac{kr(5rk+3)}{2}\text{ for $1\leq i\leq k$},\\[6pt]
(3)\,\text{from $\mathcal{R}_{k+2}$ to $\mathcal{R}_{rk+k+1}$},\,\,\omega(u_{i}^{j1})=7rk+3k+1,\,\text{ for $1\leq i\leq k$ and $1\leq j\leq r$},\\[6pt]
(4)\,\text{from $\mathcal{R}_{rk+k+2}$ to $\mathcal{R}_{2rk+k+1}$},\,\,\omega(u_{i}^{j2})=10rk+3k+2\text{ for $1\leq i\leq k$, $1\leq j\leq r$}.
\end{array}
\end{equation*}

Obviously, these are four different colors derived from the labeling matrix $\boldsymbol{\mathcal{M}}_{1}^{srk}$. This means that $\boldsymbol{\mathcal{M}}_{1}^{srk}$ corresponds to a local antimagic labeling of $S_{k}\diamond rK_{2}$ for odd $r$, yielding that $\chi_{la}(S_{k}\diamond rK_{2})\leq 4$.

Next, when $r$ is even, its corresponding labeling matrix $\boldsymbol{\mathcal{M}}_{2}^{srk}$ is obtained through the following steps.

{\it Step 1}. Utilize the equalities \eqref{sk2-1} and \eqref{sk2-2} to get $\boldsymbol{a}_{j},\boldsymbol{A}_{j},\boldsymbol{b}_{j},\boldsymbol{B}_{j},\boldsymbol{B}_{r+j}$ for $1\leq j\leq r$.

{\it Step 2}. To obtain the new matrix $\boldsymbol{A}_{i}'$ and vector $\boldsymbol{b}_{i}'$ for $1\leq j< r$, transform the diagonal entries of  $\boldsymbol{A}_{i}$ and the entries of $\boldsymbol{b}_{i}$ by reversing their order. Meanwhile, invert the order of $\boldsymbol{a}_{1}$ and the diagonal entries of $\boldsymbol{B}_{1}$ and $\boldsymbol{B}_{1}'$ to obtain the new vector $\boldsymbol{a}_{1}'$ and matrices $\boldsymbol{B}_{1}'$, $\boldsymbol{B}_{r+1}'$, respectively.

{\it Step 3}. Convert $\boldsymbol{a}_{r}$ into a matrix representation $\boldsymbol{A}_{1}'$ by arranging its entries in reverse order along the diagonal. Conversely, vectorize $\boldsymbol{A}_{1}'$ into $\boldsymbol{a}_{r}'$ by reversing the diagonal elements to form $\boldsymbol{a}_{r}'$.

The resulting labeling matrix $\boldsymbol{\mathcal{M}}_{2}^{srk}$ is structured as
\begin{equation*}
 \boldsymbol{\mathcal{M}}_{2}^{srk}=\begin{pmatrix}
    \boldsymbol{M}_{1} &  \boldsymbol{M}_{2}' &  \boldsymbol{M}_{3}'\\[3pt]
    \boldsymbol{M}_{2}'^{\rm{\,T}} & \bigstar & \boldsymbol{M}_{4}'\\[3pt]
   \boldsymbol{M}_{3}'^{\rm{\,T}}&  \boldsymbol{M}_{4}' & \bigstar
  \end{pmatrix},
\end{equation*}
where $\boldsymbol{M}_{1}\in \mathbb{R}^{(k+1)\times (k+1)}$, $\boldsymbol{M}_{2}',\boldsymbol{M}_{3}'\in\mathbb{R}^{(k+1)\times kr}$, and $\boldsymbol{M}_{4}'\in\mathbb{R}^{kr\times kr}$ is a diagonal matrix. Matrices $\boldsymbol{M}_{2}'^{\rm{\,T}},\boldsymbol{M}_{3}'^{\rm{\,T}}$ are the transpose of $\boldsymbol{M}_{2}',\boldsymbol{M}_{3}'$, respectively. The entries in $\boldsymbol{M}_{1}$ remain unchanged, and
\begin{align*}
&{\small \boldsymbol{M}_{2}' = \begin{pmatrix}
 \boldsymbol{a}_{1}'& \boldsymbol{a}_{2} & \cdots & \boldsymbol{a}_{r-1}& \boldsymbol{a}_{r}'\\
    %a1=(n+1,...,2n) to a_{1}'=(2n,...,n+1)
 \boldsymbol{A}_{r}& \boldsymbol{A}'_{r-1}&\cdots & \boldsymbol{A}'_{2}&\boldsymbol{A}_{1}'
 \end{pmatrix}},
\quad {\small  \boldsymbol{M}_{3}'  = \begin{pmatrix}
 \boldsymbol{b}_{r}& \boldsymbol{b}'_{r-1} & \cdots & \boldsymbol{b}'_{2}& \boldsymbol{b}'_{1}\\
 \boldsymbol{B}_{1}'& \boldsymbol{B}_{2}&\cdots &\boldsymbol{B}_{r-1}&\boldsymbol{B}_{r}
 \end{pmatrix} }, \\[4pt]
 &{\small \boldsymbol{M}_{4}'   = \begin{pmatrix}
\boldsymbol{B}_{r+1}'& \bigstar & \cdots & \bigstar& \bigstar\\
\bigstar& \boldsymbol{B}_{r+2}&\cdots & \bigstar&\bigstar\\
\vdots &\vdots &\ddots &\vdots&\vdots \\
 \bigstar & \bigstar &\cdots&\boldsymbol{B}_{2r-1} & \bigstar\\
 \bigstar& \bigstar& \cdots & \bigstar& \boldsymbol{B}_{2r}
    \end{pmatrix}}.
\end{align*}

It is observed that $\boldsymbol{M}_{2}'$, $\boldsymbol{M}_{3}'$, and $\boldsymbol{M}_{4}'$ are permutations of the entries of $\boldsymbol{M}_{2}$, $\boldsymbol{M}_{3}$ and $\boldsymbol{M}_{4}$, respectively, with modifications as detailed in {\it Step 2} and {\it Step 3}.

The distinct colors are then induced from the labeling matrix $\boldsymbol{\mathcal{M}}_{2}^{srk}$ by calculating the row sums as follows:
\begin{equation*}
\begin{array}{l}
(1)\,\text{for $\mathcal{R}_{1}$},\,\,\omega(c)=2k^{2}+\frac{kr(5kr+4k+3)}{2},\\[6pt]
(2)\,\text{from $\mathcal{R}_{2}$ to $\mathcal{R}_{k+1}$},\,\,\omega(v_{i})=2kr-k+1+\frac{11kr^{2}+r}{2}\text{ for $1\leq i\leq k$},\\[6pt]
(3)\,\text{from $\mathcal{R}_{k+2}$ to $\mathcal{R}_{rk+k+1}$},\,\,\omega(u_{i}^{j1})=7rk+3k+1,\,\text{ for $1\leq i\leq k$ and $1\leq j\leq r$},\\[6pt]
(4)\,\text{from $\mathcal{R}_{rk+k+2}$ to $\mathcal{R}_{2rk+k+1}$},\,\,\omega(u_{i}^{j2})=10rk+3k+2\text{ for $1\leq i\leq k$, $1\leq j\leq r$}.
\end{array}
\end{equation*}
These row sums confirm that $\boldsymbol{\mathcal{M}}_{2}^{srk}$ corresponds to a local antimagic labeling of the graph satisfying $\chi_{la}(S_{k}\diamond rK_{2}) \leq 4$ for even $r$.

In summary, we have demonstrated that $\chi_{la}(S_{k}\diamond rK_{2}) = 4$, thereby concluding the proof.
\end{proof}

\begin{example}\label{E3}
 For $k=7,r=5$, the local antimagic chromatic number is $\chi_{la}(S_{7}\diamond 5K_{2})=4$.
\end{example}	
The local antimagic labeling matrix for the graph $S_{7} \diamond 5K_{2}$ is defined as follows:
\begin{equation*}
\boldsymbol{\mathcal{M}}_{S_{7}\diamond 5K_{2}}=\left(\begin{array}{ccc}
\boldsymbol{M}_{1} & \boldsymbol{M}_{2} & \boldsymbol{M}_{3}\\[3pt]
\boldsymbol{M}_{2}^{\rm{T}} &\bigstar & \boldsymbol{M}_{4}\\[3pt]
\boldsymbol{M}_{3}^{\rm{T}} & \boldsymbol{M}_{4} &\bigstar
\end{array}\right).
\end{equation*}
Details of the blocked matrices are provided in Appendix \ref{B}.

By analyzing the matrix $\boldsymbol{\mathcal{M}}_{S_{7}\diamond 5K_{2}}$, the row sums are calculated as follows:
\vspace{-1em}
\begin{align*}
&(1)\,\text{for $\mathcal{R}_{1}$,}\quad \omega(c)=3633,\\[-4pt]
&(2)\,\text{from $\mathcal{R}_{2}$ to $\mathcal{R}_{8}$,} \quad\omega(v_{i})=1039,\quad 1\leq i\leq 7,\quad\\[-4pt]
&(3)\,\text{from $\mathcal{R}_{9}$ to $\mathcal{R}_{43}$,}\quad\omega(u_{i}^{j1})=267,\quad1\leq i\leq 7,1\leq j\leq 5,\quad\\[-4pt]
&(4)\,\text{from $\mathcal{R}_{44}$ to $\mathcal{R}_{78}$,} \quad\omega(u_{i}^{j2})=373,\quad 1\leq i\leq 7,1\leq j\leq 5.\hspace{8em}
\end{align*}\\[-35pt]
Thus we conclude that $\chi_{la}(S_{7}\diamond 5K_{2})=4$.

For the double star $S_{k_{1},k_{2}}$, the edge-corona product graph $S_{k_{1},k_{2}}\diamond rK_{2}$ is defined with its vertex and edge sets expanded as
 $V(S_{k_{1},k_{2}}\diamond rK_{2})=V(S_{k_{1},k_{2}})\cup V_{1}$ and $E(S_{k_{1},k_{2}}\diamond rK_{2})=E(S_{k_{1},k_{2}})\cup E_{1}\cup E_{2} \cup E_{3}$, respectively,
where $V_{1}=\{u_{i}^{j1},u_{i}^{j2}|1\leq i\leq k_{1}+k_{2}+1,1\leq j\leq r\}$, $E_{1}=\{c_{1}u_{1}^{js},c_{2}u_{1}^{js},u_{1}^{j1}u_{1}^{j2}|1\leq j\leq r,s=1,2\}$, $E_{2}=\{c_{1}u_{i}^{js},v_{i-1}u_{i}^{js},u_{i}^{j1}u_{i}^{j2}|2\leq i\leq k_{1}+1,s=1,2\}$ and $E_{3}=\{c_{2}u_{i}^{js},v_{i-1}u_{i}^{js},u_{i}^{j1}u_{i}^{j2}|k_{1}+2\leq i\leq k_{1}+k_{2}+1,s=1,2\}$.
Clearly, the lower bound of $\chi_{la}(S_{k_{1},k_{2}}\diamond rK_{2})$ is also 4. To find the upper bound, we construct a local antimagic labeling for this graph, leading to following theorem.
\begin{theorem}
For the double star $S_{k_{1},k_{2}}$, we have $4\leq \chi_{la}(S_{k_{1},k_{2}}\diamond rK_{2})\leq 5$.
\end{theorem}

\begin{proof}
According to the labeling matrix of the graph $S_{k}\diamond rK_{2}$ and the parity discussion of $r$, the labeling matrix of $S_{n_{1},n_{2}}\diamond rK_{2}$ is obtained by the following steps.

{\it Step 1}. Substitute $k=k_{1}+k_{2}+1$ into the equations \eqref{sk2-1} and \eqref{sk2-2} above to compute the corresponding vectors
in the equations \eqref{sk2-1} and \eqref{sk2-2} above, let $k=k_{1}+k_{2}+1$ and compute the corresponding vectors $\boldsymbol{a}_{j},\boldsymbol{b}_{j}$, and diagonal matrices $\boldsymbol{A}_{j},\boldsymbol{B}_{j},\boldsymbol{B}_{r+j}$ for $1\leq j\leq r$. If $r$ is even, apply the modifications as detailed in the proof of  Theorem \ref{The-31}.

{\it Step 2}. For each $0\leq j\leq r$, partite $\boldsymbol{a}_{j}$ into a single value $a_{j}^{1}$ and two sub-vectors $\boldsymbol{a}_{j}^{2}$, $\boldsymbol{a}_{j}^{3}$. Specifically, separate the first entry into $a_{j}^{1}$, and the next $k_{1}$ entries into vector $\boldsymbol{a}_{j}^{2}$, and the last $k_{2}$ entries into the other vector $\boldsymbol{a}_{j}^{3}$.   Apply the same partitioning method to vector $\boldsymbol{b}_{j}$ and block matrices $\boldsymbol{A}_{j}$ into a single value $A_{j}^{1}$ and two diagonal sub-matrices $\boldsymbol{A}_{j}^{2},\boldsymbol{A}_{j}^{3}$, whose entries are the middle $k_{1}$ and last $k_{2}$ diagonal entries of matrix $\boldsymbol{A}_{j}$. As well as matrix $\boldsymbol{B}_{j}$ is also partitioned accordingly.

By following these steps, the labeling matrix $\boldsymbol{\mathcal{M}}^{Dsrk}$ for $S_{k_{1},k_{2}}\diamond rK_{2}$ is constructed based on the order of vertices $c_{1},c_{2},v_{1},\cdots,v_{k_{1}},v_{k_{1}+1},\cdots,v_{k_{1}+k_{2}},\,u_{1}^{11},u_{2}^{11},\cdots,u_{k_{1}+k_{2}+1}^{11}$, $\cdots,u_{1}^{r1},u_{2}^{r1},\cdots,u_{k_{1}+k_{2}+1}^{r1},u_{1}^{12},u_{2}^{12},\cdots,u_{k_{1}+k_{2}+1}^{12},\,\cdots,u_{1}^{r2},u_{2}^{r2},\cdots,u_{k_{1}+k_{2}+1}^{r2}$. This labeling matrix is derived from the expansion of $\boldsymbol{\mathcal{M}}_{S_{k}\diamond rK_{2}}$. The specific details are as follows.

When $r$ is odd,
\begin{equation*}
\boldsymbol{\mathcal{M}}_{1}^{Dsrk}=\begin{pmatrix}
	\boldsymbol{M}_{1}^{1}  &  \boldsymbol{M}_{2}^{1} &  \boldsymbol{M}_{3}^{1} \\[3pt]
	\boldsymbol{M}_{2}^{1\,\rm{T}} & \bigstar & \boldsymbol{M}_{4}^{1} \\[3pt]
	\boldsymbol{M}_{3}^{1\,\rm{T}}&  \boldsymbol{M}_{4}^{1}  & \bigstar
\end{pmatrix},
\end{equation*}
where the blocked matrices are obtained analogously in the proof of Theorem \ref{The-31}. In detail, $\boldsymbol{M}_{1}^{1}\in \mathbb{R}^{(k_{1}+k_{2}+2)\times (k_{1}+k_{2}+2)}$, $\boldsymbol{M}_{4}^{1}\in \mathbb{R}^{r(k_{1}+k_{2}+1)\times r(k_{1}+k_{2}+1)}$, and
\begin{equation*}
	\boldsymbol{M}_{1}^{1} =\begin{pmatrix}
*& a_{0}^{1} &\boldsymbol{a}_{0}^{2\,}& \bigstar\,\\[2pt]
a_{0}^{1} &* &\bigstar&\boldsymbol{a}_{0}^{3}\\[3pt]
\boldsymbol{a}_{0}^{2\,{\rm T}}& \bigstar & \bigstar & \bigstar\\[3pt]
\bigstar&\boldsymbol{a}_{0}^{3\,{\rm T}}  & \bigstar & \bigstar
\end{pmatrix},\quad
\boldsymbol{M}_{4}^{1} =\begin{pmatrix}
\boldsymbol{B}_{r+1}& \bigstar & \cdots & \bigstar &\bigstar\\
\bigstar & \boldsymbol{B}_{r+2}& \cdots & \bigstar &\bigstar\\
\vdots & \vdots &\ddots &\vdots &\vdots \\
\bigstar & \bigstar & \cdots & \bigstar &\boldsymbol{B}_{2r}\\
\end{pmatrix}
\end{equation*}
Next, $\boldsymbol{M}_{2}^{1} =(\boldsymbol{D}^{11}_{2},\boldsymbol{D}^{21}_{2},\cdots,\boldsymbol{D}^{r1}_{2})\in\mathbb{R}^{(k_{1}+k_{2}+2)\times r(k_{1}+k_{2}+1)}$, where each sub-matrix $\boldsymbol{D}^{j1}_{2}$ is structured as
$$\boldsymbol{D}^{11}_{2}=\begin{pmatrix}
{a}_{1}^{1} &\boldsymbol{a}_{1}^{2} &\bigstar\\[2pt]
\boldsymbol{A}_{r}^{1}&\bigstar&\boldsymbol{a}_{1}^{3}\\
\bigstar &\boldsymbol{A}_{r}^{2}&\bigstar\\
\bigstar&\bigstar&\boldsymbol{A}_{r}^{3}
\end{pmatrix},
\quad\quad\cdots\cdots,\quad\quad
\boldsymbol{D}^{r1}_{2}=\begin{pmatrix}
{a}_{r}^{1} &\boldsymbol{a}_{r}^{2} &\bigstar\\[2pt]
\boldsymbol{A}_{1}^{1}&\bigstar&\boldsymbol{a}_{r}^{3}\\
\bigstar &\boldsymbol{A}_{1}^{2}&\bigstar\\
\bigstar&\bigstar&\boldsymbol{A}_{1}^{3}
\end{pmatrix}.$$
Following this way, the blocked matrix $\boldsymbol{M}_{3}^{1}\in \mathbb{R}^{(k_{1}+k_{2}+2)\times r(k_{1}+k_{2}+1)}$ is obtained in turn.
$\boldsymbol{M}_{3}^{1} =(\boldsymbol{D}^{12}_{3},\boldsymbol{D}^{22}_{3},\cdots,\boldsymbol{D}^{r2}_{3})$, where each sub-matrix $\boldsymbol{D}^{j2}_{3}$ is structured as
$$\boldsymbol{D}^{12}_{3}=\begin{pmatrix}
{b}_{r}^{1} &\boldsymbol{b}_{r}^{2} &\bigstar\\[2pt]
{B}_{1}^{1}&\bigstar&\boldsymbol{b}_{r}^{3}\\
\bigstar &\boldsymbol{B}_{1}^{2}&\bigstar\\
\bigstar&\bigstar&\boldsymbol{B}_{1}^{3}
\end{pmatrix},
\quad\quad\cdots\cdots,\quad\quad
\boldsymbol{D}^{r2}_{3}=\begin{pmatrix}
{b}_{1}^{1} &\boldsymbol{b}_{1}^{2} &\bigstar\\[2pt]
{B}_{r}^{1}&\bigstar&\boldsymbol{b}_{1}^{3}\\
\bigstar &\boldsymbol{B}_{r}^{2}&\bigstar\\
\bigstar&\bigstar&\boldsymbol{B}_{r}^{3}
\end{pmatrix}.$$

Consequently, compute the row sums of the matrix $\boldsymbol{\mathcal{M}}^{Dsrk}_{1}$ for this case, yielding the following results
\begin{equation*}
  \begin{array}{l}
(1)\text{ for $\mathcal{R}_{1}$},\,\,\omega(c_{1})=\frac{(k_{1}+1)(r-1)(tr+3t+1)}{2}+(k_{1}+1)(t+k_{1}+2),\\[6pt]
(2)\text{ for $\mathcal{R}_{2}$},\,\,\omega(c_{2})=\frac{k_{2}(r-1)(tr+3t+1)}{2}+(2r^{2}+r+1)t+k_{2}(2t+k_{1}+2),\\[6pt]
(3)\text{ from $\mathcal{R}_{3}$ to $\mathcal{R}_{t+1}$},\,\,\omega(v_{i})=\frac{r(11rt+4t+1)+t+1}{2} \,\,\text{for $1\leq i\leq t-1$},\\[6pt]
(4)\text{ From $\mathcal{R}_{t+2}$ to $\mathcal{R}_{rt+t+1}$.}\,\,\omega(u_{i}^{j1})=7rt+3t+1 \,\,\text{for $1\leq i\leq t$, $1\leq j\leq r$},\\[4pt]
(5)\text{ from $\mathcal{R}_{rt+t+2}$ to $\mathcal{R}_{2rt+t+1}$,}\,\,\omega(u_{i}^{j2})=10rt+3t+2 \,\,\text{for $1\leq i\leq t$, $1\leq j\leq r$},
\end{array}
\end{equation*}
where $t=k_{1}+k_{2}+1$. These calculations reveal that the local antimagic labeling of $S_{k_{1},k_{2}}\diamond rK_{2}$ results in five distinct colors when $r$ is odd, indicating that $\chi_{la}(S_{k_{1},k_{2}}\diamond rK_{2})\leq 5$.

When $r$ is even, the construction of the labeling matrix for $S_{k_{1},k_{2}}\diamond rK_{2}$ follows a similar procedure as outlined in  the proof of Theorem \ref{The-31}. The matrix is given by
$$\boldsymbol{\mathcal{M}}^{Dsrk}_{2}=\begin{pmatrix}
    \boldsymbol{M}_{1}^{2} &  \boldsymbol{M}_{2}^{2} &  \boldsymbol{M}_{3}^{2}\\[3pt]
    \boldsymbol{M}_{2}^{2\rm{\,T}} & \bigstar & \boldsymbol{M}_{4}^{2}\\[3pt]
   \boldsymbol{M}_{3}^{2\rm{\,T}}&  \boldsymbol{M}_{4}^{2} & \bigstar
  \end{pmatrix},$$
where $\boldsymbol{M}_{1}^{2}\in\mathbb{R}^{(k_{1}+k_{2}+2)\times (k_{1}+k_{2}+2)}$, $\boldsymbol{M}_{4}^{2}\in\mathbb{R}^{r(k_{1}+k_{2}+1)\times r(k_{1}+k_{2}+1)}$. The specific forms of these matrices are
\begin{equation*}
	\boldsymbol{M}_{1}^{2} =\begin{pmatrix}
		*& a_{0}^{1} &\boldsymbol{a}_{0}^{2\,}& \bigstar\,\\[2pt]
		a_{0}^{1} &* &\bigstar&\boldsymbol{a}_{0}^{3}\\[3pt]
		\boldsymbol{a}_{0}^{2\,{\rm T}}& \bigstar & \bigstar & \bigstar\\[3pt]
		\bigstar&\boldsymbol{a}_{0}^{3\,{\rm T}}  & \bigstar & \bigstar
	\end{pmatrix},\quad
	\boldsymbol{M}_{4}^{2} =\begin{pmatrix}
		\boldsymbol{B}_{r+1}& \bigstar & \cdots & \bigstar &\bigstar\\
		\bigstar & \boldsymbol{B}_{r+2}& \cdots & \bigstar &\bigstar\\
		\vdots & \vdots &\ddots &\vdots &\vdots \\
		\bigstar & \bigstar & \cdots & \bigstar &\boldsymbol{B}_{2r}\\
	\end{pmatrix}.
\end{equation*}
The blocked matrices $\boldsymbol{M}_{2}^{2}$ and $\boldsymbol{M}_{2}^{3}$ are
\begin{equation*}
\boldsymbol{M}_{2}^{2} ={\footnotesize{\left(\begin{array}{ccc:ccc:c:ccc:ccc}
     {a}_{1}'^{1} & \boldsymbol{a}_{1}'^{2} & \bigstar & {a}_{2}^{1} & \boldsymbol{a}_{2}^{2} &\bigstar  & \,\,\cdots\,\cdots\,\, & {a}_{r-1}^{1} & \boldsymbol{a}_{r-1}^{2}  & \bigstar  & {a}_{r}'^{1} & \boldsymbol{a}_{r}'^{2}  & \bigstar\\[6pt]
     {A}_{r}^{1} & \bigstar & \boldsymbol{a}_{1}'^{3} & {A}_{r-1}'^{1} & \bigstar & \boldsymbol{a}_{2}^{3} & \,\,\cdots\,\cdots\,\, &{A}_{2}'^{1} & \bigstar & \boldsymbol{a}_{r-1}^{3} & {A}_{1}'^{1} & \bigstar & \boldsymbol{a}_{r}^{3}\\[6pt]
    %a1=(n+1,...,2n) to a_{1}'=(2n,...,n+1)
    \bigstar &\boldsymbol{A}_{r}^{2}&\bigstar &\bigstar &\boldsymbol{A}_{r-1}'^{2}&\bigstar & \,\,\cdots\,\cdots\,\, & \bigstar &\boldsymbol{A}_{2}'^{2}&\bigstar & \bigstar &\boldsymbol{A}_{1}'^{2}&\bigstar\\[6pt]
    \bigstar&\bigstar&\boldsymbol{A}_{r}^{3} &\bigstar&\bigstar&\boldsymbol{A}_{r-1}'^{3}& \,\,\cdots\,\cdots\,\, &\bigstar&\bigstar&\boldsymbol{A}_{2}'^{3}&\bigstar&\bigstar&\boldsymbol{A}_{1}'^{3}
     %& \boldsymbol{A}_{r}& \boldsymbol{A}'_{r-1}&\cdots & \boldsymbol{A}'_{2}&\boldsymbol{A}_{1}'
    \end{array}\right)}},
    % convert a_{r} to A', and convert the original A_{1}' to a_{r}'
\end{equation*}
\begin{equation*}
\boldsymbol{M}_{3}^{2} ={\footnotesize{\left(\begin{array}{ccc:ccc:c:ccc:ccc}
     {b}_{r}^{1} & \boldsymbol{b}_{r}^{2} & \bigstar & {b}_{r-1}'^{1} & \boldsymbol{b}_{r-1}'^{2} &\bigstar  & \,\,\cdots\,\cdots\,\, & {b}_{2}'^{1} & \boldsymbol{b}_{2}'^{2}  & \bigstar  & {b}_{1}'^{1} & \boldsymbol{b}_{1}'^{2}  & \bigstar\\[6pt]
     {B}_{1}'^{1} & \bigstar & \boldsymbol{b}_{r}^{3} & {B}_{2}^{1} & \bigstar & \boldsymbol{b}_{r-1}'^{3} & \,\,\cdots\,\cdots\,\, &{B}_{r-1}^{1} & \bigstar & \boldsymbol{b}_{2}'^{3} & {B}_{r}^{1} & \bigstar & \boldsymbol{b}_{1}'^{3}\\[6pt]
    %a1=(n+1,...,2n) to a_{1}'=(2n,...,n+1)
    \bigstar &\boldsymbol{B}_{1}'^{2}&\bigstar & \bigstar &\boldsymbol{B}_{2}^{2}&\bigstar & \,\,\cdots\,\cdots\,\, & \bigstar &\boldsymbol{B}_{r-1}^{2}&\bigstar & \bigstar &\boldsymbol{B}_{r}^{2}&\bigstar \\[6pt]
    \bigstar&\bigstar&\boldsymbol{B}_{1}'^{3} &\bigstar&\bigstar&\boldsymbol{B}_{2}^{3}& \,\,\cdots\,\cdots &\bigstar&\bigstar&\boldsymbol{B}_{r-1}^{3} & \bigstar&\bigstar&\boldsymbol{B}_{r}^{3}
     %& \boldsymbol{A}_{r}& \boldsymbol{A}'_{r-1}&\cdots & \boldsymbol{A}'_{2}&\boldsymbol{A}_{1}'
    \end{array}\right)}}.
    % convert a_{r} to A', and convert the original A_{1}' to a_{r}'
\end{equation*}
We obtain the row sums of the labeling matrix $\boldsymbol{\mathcal{M}}^{Dsrk}_{2}$ for even $r$ as follow:
\begin{equation*}
  \begin{array}{l}
(1)\text{ for $\mathcal{R}_{1}$},\quad\omega(c_{1})=(k_{1}+1)(2rt-t+3k_{1}+3)+\frac{r(k_{1}+1)(5rt+3)}{2},\\[6pt]
(2)\text{ for $\mathcal{R}_{2}$},\quad\omega(c_{2})=k_{2}(2rt+5t-3k_{2})-t+1+\frac{rk_{2}(5rt+3)+r(11rt+4t+1)}{2},\\[6pt]
(3)\text{ from $\mathcal{R}_{3}$ to $\mathcal{R}_{t+1}$},\quad\omega(v_{i})=\frac{r(11rt+4t+1)}{2}+1-t\text{ for $1\leq i\leq t-1$},\\[6pt]
(4)\text{ from $\mathcal{R}_{t+2}$ to $\mathcal{R}_{rt+t+1}$,}\quad\omega(u_{i}^{j1})=7rt+3t+1\text{ for $1\leq i\leq t$ and $1\leq j\leq r$},\\[4pt]
(5)\text{ from $\mathcal{R}_{rt+t+2}$ to $\mathcal{R}_{2rt+t+1}$,}\quad\omega(u_{i}^{j2})=10rt+3t+2\text{ for $1\leq i\leq t$ and $1\leq j\leq r$},
\end{array}
\end{equation*}
where $t=k_{1}+k_{2}+1$. The labeling matrix $\boldsymbol{\mathcal{M}}^{Dsrk}_{2}$ yields five distinct values for the row sums, which correspond to a local antimagic labeling of $S_{k_{1},k_{2}}\diamond rK_{2}$ utilizing five colors. Consequently, $\chi_{la}(S_{k_{1},k_{2}}\diamond rK_{2})\leq 5$. This completes proof.
\end{proof}

\begin{example}\label{E4}
 For $k_{1}=3,k_{2}=4,r=6$, we can determine that $4\leq\chi_{la}(S_{3,4}\diamond 6K_{2})\leq 5$.
\end{example}
The local antimagic labeling matrix of $S_{3,4}\diamond 6K_{2}$ is
\begin{equation*}
\boldsymbol{\mathcal{M}}_{S_{3,4}\diamond 6K_{2}}=\begin{pmatrix}
		\boldsymbol{M}_{1} &  \boldsymbol{M}_{2} &  \boldsymbol{M}_{3}\\[3pt]
		\boldsymbol{M}_{2}^{\rm{T}} & \bigstar & \boldsymbol{M}_{4}\\[3pt]
		\boldsymbol{M}_{3}^{\rm{T}}&  \boldsymbol{M}_{4} & \bigstar
	\end{pmatrix}.
\end{equation*}
Refer to Appendix \ref{C} for detailed information on the blocked matrices in $\boldsymbol{\mathcal{M}}_{S_{3,4}\diamond 6K_{2}}$.

Accordingly, the row sums of $\boldsymbol{\mathcal{M}}_{S_{3,4}\diamond 6K_{2}}$ are
\begin{equation*}
\begin{array}{l}
(1)\,\text{for $\mathcal{R}_{1}$,} \quad\omega(c_{1})=3316, \quad\quad(2)\,\text{for $\mathcal{R}_{2}$,} \quad\omega(c_{2})=5088,\\[4pt]
(3)\,\text{from $\mathcal{R}_{3}$ to $\mathcal{R}_{9}$,} \quad\omega(v_{i})=1676,\quad 1\leq i\leq 7,\\[4pt]
(4)\,\text{from $\mathcal{R}_{10}$ to $\mathcal{R}_{57}$,} \quad\omega(u_{i}^{j1})=345 ,\quad 1\leq i\leq 7, 1\leq j\leq 6,\hspace{7.5em}\\[4pt]
(5)\,\text{from $\mathcal{R}_{58}$ to $\mathcal{R}_{105}$, }
\quad\omega(u_{i}^{j2})=506, \quad 1\leq i\leq 7, 1\leq j\leq 6.
\end{array}
\end{equation*}
Based on this labeling matrix $\boldsymbol{\mathcal{M}}_{S_{3,4}\diamond 6K_{2}}$, a local antimagic labeling of $S_{3,4}\diamond 6K_{2}$ using five colors is achieved.  Therefore, we have $4\leq\chi_{la}(S_{3,4}\diamond 6K_{2})\leq 5$.

\section{Conclusion}
This paper presents the local antimagic (total) chromatic numbers for a specific graph known as the firecracker graph $F_{n,k}$. Currently, the investigation of local antimagic chromatic numbers in the context of graphic operations primarily focuses on the join graph, the corona product and the lexicographic product of two graphs. Building upon this foundation, this paper shows the edge-corona product operation of two graphs, $G$ and $H$, and provides the local antimagic chromatic number associated with this graphic operation. In this context, graph $G$ is represented by a star $S_{k}$ or a double star $S_{k_{1},k_{2}}$, while graph $H$ is the empty graph $\overline{K_{r}}$ or the complete graph $K_2$. Moreover, the local antimagic (total) chromatic numbers for other graph families and graphic operations will be investigated in the next step.

\section{Declaration of Competing Interest}
The authors declare that they have  no known competing financial interests or personal relationships that could have appeared to influence the work reported in this paper.
%\begin{align*}

\section{Acknowledgements}
This work was conducted by Xue Yang during her master's program. This project has received funding from 2023 Xinjiang Natural Science Foundation General Project No. 2023D01A36; and 2023 Xinjiang Natural Science Foundation For Youths No. 2023D01B48. Hong Bian's work is partially supported by National Natural Science Foundation of China No. 12361072. We are grateful to Professor Xueliang Li for his meticulous revisions and polishing of the paper, which significantly improved its readability.

\section*{Appendix}
\appendix
\section{The matrices of Example \ref{E2}}\label{A}
Here, 
\begin{equation*}
	\boldsymbol{M}_{1}=\begin{pmatrix}
		* &b_{0}^{1} &\boldsymbol{b}_{0}^{2} & \bigstar\\
		b_{0}^{1} & * & \bigstar &\boldsymbol{b}_{0}^{3}\\[3pt]
		\boldsymbol{b}_{0}^{2\,{\rm T}} & \bigstar & \bigstar & \bigstar\\
		\bigstar & \boldsymbol{b}_{0}^{3\,{\rm T}} & \bigstar & \bigstar	
	\end{pmatrix},
	\quad {\rm where}\quad
	\begin{array}{l}
		{b}_{0}^{1}=1,\\[3pt]
		\boldsymbol{b}_{0}^{2}=(2,3,4,5),\\[3pt]
		\boldsymbol{b}_{0}^{3}=(6,7,8,9,10).
	\end{array}
\end{equation*}
%where ${b}_{0}^{1}=1$, $\boldsymbol{b}_{0}^{2}=(2,3,4,5)$, and $\boldsymbol{b}_{0}^{3}=(6,7,8,9,10)$.

And
\begin{equation*}
	\boldsymbol{M}_{2}
	%=(\boldsymbol{M}_{2}^{1},\boldsymbol{M}_{2}^{2},\boldsymbol{M}_{2}^{3},\boldsymbol{M}_{2}^{4})
	=\left(\begin{array}{ccc:ccc:ccc:ccc}
		{b}_{1}^{1}	& \boldsymbol{b}_{1}^{2} & \bigstar &{b}_{2}^{1}	& \boldsymbol{b}_{2}^{2} & \bigstar &{b}_{3}^{1}	& \boldsymbol{b}_{3}^{2} & \bigstar &{b}_{4}^{1}&\boldsymbol{b}_{4}^{2} & \bigstar\\
		%%two row
		{B}_{4}^{1} & \bigstar & \boldsymbol{b}_{1}^{3}
		&{B}_{3}^{1} & \bigstar & \boldsymbol{b}_{2}^{3}
		&{B}_{2}^{1} & \bigstar & \boldsymbol{b}_{3}^{3}
		&{B}_{1}^{1} & \bigstar & \boldsymbol{b}_{4}^{3}\\
		%%three row
		\bigstar & \boldsymbol{B}_{4}^{2} & \bigstar
		&\bigstar & \boldsymbol{B}_{3}^{2} & \bigstar
		&\bigstar & \boldsymbol{B}_{2}^{2} & \bigstar
		&\bigstar & \boldsymbol{B}_{1}^{2} & \bigstar\\
		%%four row
		\bigstar &\bigstar & \boldsymbol{B}_{4}^{3}
		&\bigstar &\bigstar & \boldsymbol{B}_{3}^{3}
		&\bigstar &\bigstar & \boldsymbol{B}_{2}^{3}
		&\bigstar &\bigstar & \boldsymbol{B}_{1}^{3}
	\end{array}\right),
\end{equation*}
where ${b}_{1}^{1}=20$, ${b}_{2}^{1}=31$, ${b}_{3}^{1}=42$, and ${b}_{4}^{1}=80$. And ${B}_{4}^{1}=81$, ${B}_{3}^{1}=70$, ${B}_{2}^{1}=59$, and ${B}_{1}^{1}=21$. In addition, the vectors in $\boldsymbol{M}_{2}$ are as follow:
\begin{equation*}
	\begin{array}{ll}
		\boldsymbol{b}_{1}^{2}=(19,18,17,16), &\boldsymbol{b}_{2}^{2}=(32,33,34,35),\\[4pt]
		\boldsymbol{b}_{3}^{2}=(44,46,48,50),&\boldsymbol{b}_{4}^{2}=(79,78,77,76),\\[4pt]
		\boldsymbol{b}_{1}^{3}=(15,14,13,12,11),&\boldsymbol{b}_{2}^{3}=(36,37,38,39,40),\\[4pt]
		\boldsymbol{b}_{3}^{3}=(52,54,56,58,60),
		&\boldsymbol{b}_{4}^{3}=(75,74,73,72,71).
	\end{array}
\end{equation*}
The diagonal matrices in $\boldsymbol{M}_{2}$ are as follow:
\begin{equation*}
	\begin{array}{ll}
		\boldsymbol{B}_{4}^{2}={\rm\bf diag}(82,83,84,85), &\boldsymbol{B}_{3}^{2}={\rm\bf diag}(69,68,67,66),\\[4pt]
		\boldsymbol{B}_{2}^{2}={\rm\bf diag}(57,55,53,51),
		&\boldsymbol{B}_{1}^{2}={\rm\bf diag}(22,23,24,25),\\[4pt]
		%% three
		\boldsymbol{B}_{4}^{3}={\rm\bf diag}(86,87,88,89,90),
		&\boldsymbol{B}_{3}^{3}={\rm\bf diag}(65,64,63,62,61),\\[4pt]
		\boldsymbol{B}_{2}^{3}={\rm\bf diag}(49,47,45,43,41),
		&\boldsymbol{B}_{1}^{3}={\rm\bf diag}(26,27,28,29,30).
	\end{array}
\end{equation*}

\section{The matrices of Example \ref{E3}}\label{B}
Here, the matrix $\boldsymbol{M}_{1}$ is
\begin{equation*}
	\boldsymbol{M}_{1}=\begin{pmatrix}
		* & \boldsymbol{a}_{0}\\
		\boldsymbol{a}_{0}^{{\rm T}} & \bigstar
	\end{pmatrix},
	\quad {\rm where}\quad\boldsymbol{a}_{0}=(1,2,3,4,5,6,7).
\end{equation*}
The corresponding matrix $\boldsymbol{M}_{2}$ is
\begin{equation*}
	\boldsymbol{M}_{2}=\begin{pmatrix}
		\boldsymbol{a}_{1} & \boldsymbol{a}_{2} & \boldsymbol{a}_{3} & \boldsymbol{a}_{4} & \boldsymbol{a}_{5}\\
		\boldsymbol{A}_{5} & \boldsymbol{A}_{4} & \boldsymbol{A}_{3} & \boldsymbol{A}_{2} & \boldsymbol{A}_{1}
	\end{pmatrix},
\end{equation*}
where
\begin{equation*}
	\small
	\begin{array}{ll}
		\boldsymbol{a}_{1}=(8,9,10,11,12,13,14),
		& \boldsymbol{A}_{5}={\rm\bf diag}(111,109,107,105,103,101,99),\\[4pt]
		\boldsymbol{a}_{2}=(21,20,19,18,17,16,15), &\boldsymbol{A}_{4}={\rm\bf diag}(85,87,89,91,93,95,97),\\[4pt] 
		\boldsymbol{a}_{3}=(22,23,24,25,26,27,28), 
		&\boldsymbol{A}_{3}={\rm\bf diag}(83,81,79,77,75,73,71),\\[4pt]
		\boldsymbol{a}_{4}=(35,34,33,32,31,30,29),
		&\boldsymbol{A}_{2}={\rm\bf diag}(57,59,61,63,65,67,69),\\[4pt] 
		\boldsymbol{a}_{5}=(36,37,38,39,40,41,42),
		&\boldsymbol{A}_{1}={\rm\bf diag}(55,53,51,49,47,45,43).
	\end{array}
\end{equation*}
Successively, the corresponding matrix $\boldsymbol{M}_{3}$ is
\begin{equation*}
	\boldsymbol{M}_{3}=\begin{pmatrix}
		\boldsymbol{b}_{5} & \boldsymbol{b}_{4} & \boldsymbol{b}_{3} & \boldsymbol{b}_{2} & \boldsymbol{b}_{1}\\
		\boldsymbol{B}_{1} & \boldsymbol{B}_{2} & \boldsymbol{B}_{3} & \boldsymbol{B}_{4} & \boldsymbol{B}_{5}
	\end{pmatrix},
\end{equation*}
where
\begin{equation*}
	\small
	\begin{array}{ll}
		\boldsymbol{b}_{5}=(112,110,108,106,104,102,100),
		& \boldsymbol{B}_{1}={\rm\bf diag}(113,114,115,116,117,118,119),\\[4pt]
		\boldsymbol{b}_{4}=(86,88,90,92,94,96,98), &\boldsymbol{B}_{2}={\rm\bf diag}(126,125,124,123,122,121,120),\\[4pt] 
		\boldsymbol{b}_{3}=(84,82,80,78,76,74,72), 
		&\boldsymbol{B}_{3}={\rm\bf diag}(127,128,129,130,131,132,133),\\[4pt]
		\boldsymbol{b}_{2}=(58,60,62,64,66,68,70),
		&\boldsymbol{B}_{4}={\rm\bf diag}(140,139,138,137,136,135,134),\\[4pt] 
		\boldsymbol{b}_{1}=(56,54,52,50,48,46,44),
		&\boldsymbol{B}_{5}={\rm\bf diag}(141,142,143,144,145,146,147).
	\end{array}
\end{equation*}
Finally, $\boldsymbol{M}_{4}$ is diagonal matrix ${\rm\bf diag}(\boldsymbol{B}_{6},\boldsymbol{B}_{7},\boldsymbol{B}_{8},\boldsymbol{B}_{9},\boldsymbol{B}_{10})$, where
%\begin{equation*}
%	\boldsymbol{M}_{4}=\begin{pmatrix}
	%		\boldsymbol{B}_{6} & &&&\\
	%		&\boldsymbol{B}_{7}  &&&\\
	%		&&\boldsymbol{B}_{8}  &&\\
	%		&&&\boldsymbol{B}_{9}  &\\
	%		&&&&\boldsymbol{B}_{10}  \\
	%	\end{pmatrix},
%\end{equation*}
\begin{equation*}
	\small
	\begin{array}{ll}
		\boldsymbol{B}_{6}={\rm\bf diag}(148,149,150,151,152,153,154),
		& \boldsymbol{B}_{7}={\rm\bf diag}(161,160,159,158,157,156,155),\\[4pt]
		\boldsymbol{B}_{8}={\rm\bf diag}(162,163,164,165,166,167,168), &\boldsymbol{B}_{9}={\rm\bf diag}(175,174,173,172,171,170,169),\\[4pt] 
		\boldsymbol{B}_{10}={\rm\bf diag}(176,177,178,179,180,181,182).
	\end{array}
\end{equation*}

\section{The labeling matrices of Example \ref{E4}}\label{C}
Here, the matrix $\boldsymbol{M}_{1}$ is
\begin{equation*}
	\boldsymbol{M}_{1}=\begin{pmatrix}
		* & a_{0}^{1} & \boldsymbol{a}_{0}^{2} &\bigstar\\[3pt]
		a_{0}^{1} &	*&\bigstar &  \boldsymbol{a}_{0}^{3} \\[3pt]
		\boldsymbol{a}_{0}^{2\,{\rm T}} &	\bigstar&\bigstar & \bigstar\\[3pt]
		\bigstar &\boldsymbol{a}_{0}^{3\,{\rm T}} &\bigstar&\bigstar
	\end{pmatrix},
	\quad{\rm where}\quad
	\begin{array}{l}
	a_{0}^{1}=1,\\[4pt]
	\boldsymbol{a}_{0}^{2}=(2,3,4),\\[4pt]
	\boldsymbol{a}_{0}^{3}=(5,6,7,8).
	\end{array}
\end{equation*}
The block matrix $\boldsymbol{M}_{2}$ is constructed by 
\begin{equation*}
	\small
\left(\begin{array}{ccc:ccc:ccc:ccc:ccc:ccc}
%% one
a_{1}^{1} & \boldsymbol{a}_{1}^{2} &\bigstar
& a_{2}^{1} & \boldsymbol{a}_{2}^{2} &\bigstar
&a_{3}^{1} & \boldsymbol{a}_{3}^{2} &\bigstar
&a_{4}^{1} & \boldsymbol{a}_{4}^{2} &\bigstar
&a_{5}^{1} & \boldsymbol{a}_{5}^{2} &\bigstar
&a_{6}^{1} & \boldsymbol{a}_{6}^{2} &\bigstar\\[4pt]
%%two
{A}_{6}^{1} &\bigstar & \boldsymbol{a}_{1}^{3}
&{A}_{5}^{1} &\bigstar & \boldsymbol{a}_{2}^{3}
&{A}_{4}^{1} &\bigstar & \boldsymbol{a}_{3}^{3}	
&{A}_{3}^{1} &\bigstar & \boldsymbol{a}_{4}^{3}
&{A}_{2}^{1} &\bigstar & \boldsymbol{a}_{5}^{3}
&{A}_{1}^{1} &\bigstar & \boldsymbol{a}_{6}^{3}\\[4pt]
%% three
\bigstar&\boldsymbol{A}_{6}^{2}&\bigstar
&\bigstar&\boldsymbol{A}_{5}^{2}&\bigstar
&\bigstar&\boldsymbol{A}_{4}^{2}&\bigstar
&\bigstar&\boldsymbol{A}_{3}^{2}&\bigstar
&\bigstar&\boldsymbol{A}_{2}^{2}&\bigstar
&\bigstar&\boldsymbol{A}_{1}^{2}&\bigstar\\[4pt]
%%four
\bigstar&\bigstar&\boldsymbol{A}_{6}^{3}
&\bigstar&\bigstar&\boldsymbol{A}_{5}^{3}
&\bigstar&\bigstar&\boldsymbol{A}_{4}^{3}
&\bigstar&\bigstar&\boldsymbol{A}_{3}^{3}
&\bigstar&\bigstar&\boldsymbol{A}_{2}^{3}
&\bigstar&\bigstar&\boldsymbol{A}_{1}^{3}
\end{array}\right)
\end{equation*}
where $a_{1}^{1}=16$, $a_{2}^{1}=24$, $a_{3}^{1}=25$, $a_{4}^{1}=40$, $a_{5}^{1}=41$, and $a_{6}^{1}=57$. And ${A}_{6}^{1}=121$, ${A}_{5}^{1}=137$, ${A}_{4}^{1}=119$, ${A}_{3}^{1}=89$, ${A}_{2}^{1}=87$, and ${A}_{1}^{1}=56$. The vectors in $\boldsymbol{M}_{2}$ are as follow:
\begin{equation*}
\begin{array}{ll}
\boldsymbol{a}_{1}^{2}=(15,14,13), 
&\boldsymbol{a}_{1}^{3}=(12,11,10,9),\\[4pt]
\boldsymbol{a}_{2}^{2}=(23,22,21), 
&\boldsymbol{a}_{2}^{3}=(20,19,18,17),\\[4pt]
\boldsymbol{a}_{3}^{2}=(26,27,28), 
&\boldsymbol{a}_{3}^{3}=(29,30,31,32),\\[4pt]
\boldsymbol{a}_{4}^{2}=(39,38,37), 
&\boldsymbol{a}_{4}^{3}=(36,35,34,33),\\[4pt]
\boldsymbol{a}_{5}^{2}=(42,43,44), 
&\boldsymbol{a}_{5}^{3}=(45,46,47,48),\\[4pt]
\boldsymbol{a}_{6}^{2}=(59,61,63), 
&\boldsymbol{a}_{6}^{3}=(65,67,69,71).
\end{array}
\end{equation*}
Then the diagonal matrices in $\boldsymbol{M}_{2}$ are as follows:
\begin{equation*}
	\begin{array}{ll}
		\boldsymbol{A}_{6}^{2}={\rm\bf diag}(123,125,127), 
		&\boldsymbol{A}_{6}^{3}={\rm\bf diag}(129,131,133,135),\\[4pt]
		\boldsymbol{A}_{5}^{2}={\rm\bf diag}(139,141,143), 
		&\boldsymbol{A}_{5}^{3}={\rm\bf diag}(145,147,149,151),\\[4pt]
		\boldsymbol{A}_{4}^{2}={\rm\bf diag}(117,115,113), 
		&\boldsymbol{A}_{4}^{3}={\rm\bf diag}(111,109,107,105),\\[4pt]
		\boldsymbol{A}_{3}^{2}={\rm\bf diag}(91,93,95), 
		&\boldsymbol{A}_{3}^{3}={\rm\bf diag}(97,99,101,103),\\[4pt]
		\boldsymbol{A}_{2}^{2}={\rm\bf diag}(85,83,81), 
		&\boldsymbol{A}_{2}^{3}={\rm\bf diag}(79,77,75,73),\\[4pt]
		\boldsymbol{A}_{1}^{2}={\rm\bf diag}(55,54,53), 
		&\boldsymbol{A}_{1}^{3}={\rm\bf diag}(52,51,50,49).
	\end{array}
\end{equation*}
The corresponding block matrix $\boldsymbol{M}_{3}$ is shown in
\begin{equation*}
	\small
	\left(\begin{array}{ccc:ccc:ccc:ccc:ccc:ccc}
		%% one
		b_{6}^{1} & \boldsymbol{b}_{6}^{2} &\bigstar
		&b_{5}^{1} & \boldsymbol{b}_{5}^{2} &\bigstar
		&b_{4}^{1} & \boldsymbol{b}_{4}^{2} &\bigstar
		&b_{3}^{1} & \boldsymbol{b}_{3}^{2} &\bigstar
		&b_{2}^{1} & \boldsymbol{b}_{2}^{2} &\bigstar
		&b_{1}^{1} & \boldsymbol{b}_{1}^{2} &\bigstar\\[4pt]
		%%two
		{B}_{1}^{1} &\bigstar & \boldsymbol{b}_{6}^{3}
		&{B}_{2}^{1} &\bigstar & \boldsymbol{b}_{5}^{3}
		&{B}_{3}^{1} &\bigstar & \boldsymbol{b}_{4}^{3}
		&{B}_{4}^{1} &\bigstar & \boldsymbol{b}_{3}^{3}
		&{B}_{5}^{1} &\bigstar & \boldsymbol{b}_{2}^{3}
		&{B}_{6}^{1} &\bigstar & \boldsymbol{b}_{1}^{3}
		\\[4pt]
		%% three
		\bigstar&\boldsymbol{B}_{1}^{2}&\bigstar
		&\bigstar&\boldsymbol{B}_{2}^{2}&\bigstar
		&\bigstar&\boldsymbol{B}_{3}^{2}&\bigstar
		&\bigstar&\boldsymbol{B}_{4}^{2}&\bigstar
		&\bigstar&\boldsymbol{B}_{5}^{2}&\bigstar
		&\bigstar&\boldsymbol{B}_{6}^{2}&\bigstar\\[4pt]
		%%four
		\bigstar&\bigstar&\boldsymbol{B}_{1}^{3}
		&\bigstar&\bigstar&\boldsymbol{B}_{2}^{3}
		&\bigstar&\bigstar&\boldsymbol{B}_{3}^{3}
		&\bigstar&\bigstar&\boldsymbol{B}_{4}^{3}
		&\bigstar&\bigstar&\boldsymbol{B}_{5}^{3}
		&\bigstar&\bigstar&\boldsymbol{B}_{6}^{3}
	\end{array}\right)
\end{equation*}
where $b_{6}^{1}=138$, $b_{5}^{1}=122$, $b_{4}^{1}=120$, $b_{3}^{1}=90$, $b_{2}^{1}=88$, and $b_{1}^{1}=58$. And ${B}_{1}^{1}=160$, ${B}_{2}^{1}=168$, ${B}_{3}^{1}=169$, ${B}_{4}^{1}=184$, ${B}_{5}^{1}=185$, and ${B}_{6}^{1}=200$. The vectors in $\boldsymbol{M}_{3}$ are as follow:
\begin{equation*}
	\begin{array}{ll}
		\boldsymbol{b}_{6}^{2}=(140,142,144), 
		&\boldsymbol{b}_{6}^{3}=(146,148,150,152),\\[4pt]
		\boldsymbol{b}_{5}^{2}=(124,126,128), 
		&\boldsymbol{b}_{5}^{3}=(130,132,134,136),\\[4pt]
		\boldsymbol{b}_{4}^{2}=(118,116,114), 
		&\boldsymbol{b}_{4}^{3}=(112,110,108,106),\\[4pt]
		\boldsymbol{b}_{3}^{2}=(92,94,96), 
		&\boldsymbol{b}_{3}^{3}=(98,100,102,104),\\[4pt]
		\boldsymbol{b}_{2}^{2}=(86,84,82), 
		&\boldsymbol{b}_{2}^{3}=(80,78,76,74),\\[4pt]
		\boldsymbol{b}_{1}^{2}=(60,62,64), 
		&\boldsymbol{b}_{1}^{3}=(66,68,70,72).
	\end{array}
\end{equation*}
Then the diagonal matrices in $\boldsymbol{M}_{3}$ are af follows:
\begin{equation*}
	\begin{array}{ll}
		\boldsymbol{B}_{1}^{2}={\rm\bf diag}(159,158,157), 
		&\boldsymbol{B}_{1}^{3}={\rm\bf diag}(156,155,154,153),\\[4pt]
		\boldsymbol{B}_{2}^{2}={\rm\bf diag}(167,166,165), 
		&\boldsymbol{B}_{2}^{3}={\rm\bf diag}(164,163,162,161),\\[4pt]
		\boldsymbol{B}_{3}^{2}={\rm\bf diag}(170,171,172), 
		&\boldsymbol{B}_{3}^{3}={\rm\bf diag}(173,174,175,176),\\[4pt]
		\boldsymbol{B}_{4}^{2}={\rm\bf diag}(183,182,181), 
		&\boldsymbol{B}_{4}^{3}={\rm\bf diag}(180,179,178,177),\\[4pt]
		\boldsymbol{B}_{5}^{2}={\rm\bf diag}(186,187,188), 
		&\boldsymbol{B}_{5}^{3}={\rm\bf diag}(189,190,191,192),\\[4pt]
		\boldsymbol{B}_{6}^{2}={\rm\bf diag}(199,198,197), 
		&\boldsymbol{B}_{6}^{3}={\rm\bf diag}(196,195,194,193).
	\end{array}
\end{equation*}
Finally, the diagonal matrices in $\boldsymbol{M}_{4}={\rm\bf diag}(\boldsymbol{B}_{7},\boldsymbol{B}_{8},\boldsymbol{B}_{9},\boldsymbol{B}_{10},\boldsymbol{B}_{11},\boldsymbol{B}_{12})$ are as follow:
\begin{equation*}
	\begin{array}{l}
		\boldsymbol{B}_{7}={\rm\bf diag}(208,207,206,205,204,203,202,201),\\[4pt]
		\boldsymbol{B}_{8}={\rm\bf diag}(216,215,214,213,212,211,210,209),\\[4pt]
		\boldsymbol{B}_{9}={\rm\bf diag}(217,218,219,220,221,222,223,224), \\[4pt]
		\boldsymbol{B}_{10}={\rm\bf diag}(232,231,230,229,228,227,226,225),\\[4pt]
		\boldsymbol{B}_{11}={\rm\bf diag}(233,234,235,236,237,238,239,240), \\[4pt]
		\boldsymbol{B}_{12}={\rm\bf diag}(248,247,246,245,244,243,242,241).
		\end{array}
\end{equation*}

\end{document}